\documentclass[12pt,letterpaper]{article}	

\usepackage{linguex} 
\usepackage{times} 
\usepackage{lipsum} 
\usepackage{ragged2e}
\usepackage{mathtools}
\usepackage{amsmath}
\usepackage{amsthm}
\usepackage{amsfonts}
\usepackage[utf8]{inputenc}
\usepackage{graphicx}
\usepackage{algorithm}
\usepackage{algpseudocode}
\usepackage{float}
\usepackage{subfloat}
\usepackage{subfig}
\usepackage{tikz}
\usetikzlibrary{shapes}
\usepackage{bm}
\usepackage{pgfplots}
\usepackage{cases}

\definecolor{darkRed}{HTML}{960018}
\definecolor{darkBlue}{HTML}{191970}

\usepackage{scrextend} 
\deffootnote[.5em]{0em}{1em}{\textsuperscript{\thefootnotemark}\,}

\newcounter{savefootnote} 
\newcounter{symfootnote}
\newcommand{\symfootnote}[1]{%
	\setcounter{savefootnote}{\value{footnote}}%
	\setcounter{footnote}{\value{symfootnote}}%
	\ifnum\value{footnote}>8\setcounter{footnote}{0}\fi%
	\let\oldthefootnote=\thefootnote%
	\renewcommand{\thefootnote}{\fnsymbol{footnote}}%
	\footnote{#1}%
	\let\thefootnote=\oldthefootnote%
	\setcounter{symfootnote}{\value{footnote}}%
	\setcounter{footnote}{\value{savefootnote}}%
}
\newtheorem{theorem}{Theorem}
\newtheorem{lemma}{Lemma}

\newtheorem{assumption}{Assumption}

\newtheorem{remark}{Remark}

\numberwithin{equation}{section}

\usepackage[labelsep=period]{caption}

\usepackage[margin=1.0in]{geometry}
\usepackage[compact]{titlesec}
\titleformat{\section}[runin]{\normalfont\bfseries}{\thesection.}{.5em}{}[.]
\titleformat{\subsection}[runin]{\normalfont\scshape}{\thesubsection.}{.5em}{}[.]
\titleformat{\subsubsection}[runin]{\normalfont\scshape}{\thesubsubsection.}{.5em}{}[.]
\usepackage[colorlinks,allcolors={black},urlcolor={blue}]{hyperref} 

\usepackage[
backend=bibtex,
style=ieee,
citestyle=numeric,
sorting=nty,
maxbibnames=99,
mincitenames=1,
maxcitenames=2
]{biblatex}
\renewenvironment{abstract}{%
	\noindent\begin{minipage}{1\textwidth}
		\setlength{\leftskip}{0.4in}
		\setlength{\rightskip}{0.4in}
		\textbf{Abstract.}}
	{\end{minipage}}

\newcommand\blfootnote[1]{%
	\begingroup
	\renewcommand\thefootnote{}\footnote{#1}%
	\addtocounter{footnote}{-1}%
	\endgroup
}

\begin{document}	
	\setlength{\Extopsep}{6pt}
	\setlength{\Exlabelsep}{9pt} 
	
	\begin{center}
		\normalfont\bfseries
		A mixed-dimensional model for the electrostatic problem on coupled domains
		\vskip .5em
		\normalfont
		{Beatrice Crippa$^1$, Anna Scotti$^1$, Andrea Villa$^2$} \\
		\vskip .5em
		\footnotesize{$^1$ MOX-Laboratory for Modeling and Scientific Computing, Department of Mathematics, Politecnico di Milano, 20133 Milan, Italy \\
			$^2$ Ricerca Sul Sistema Energetico (RSE), 20134 Milano, Italy}
        \blfootnote{This work has been financed by the Research Found for the Italian Electrical System under the Contract Agreement between RSE and the Ministry of Economic Development.}
	\end{center}
	
	\justifying
	
	\begin{abstract}
		We derive a mixed-dimensional 3D-1D formulation of the electrostatic equation in two domains with different dielectric constants to compute, with an affordable computational cost, the electric field and potential in the relevant case of thin inclusions in a larger 3D domain. The numerical solution is obtained by Mixed Finite Elements for the 3D problem and Finite Elements on the 1D domain. We analyze some test cases with simple geometries to validate the proposed approach against analytical solutions, and perform comparisons with the fully resolved 3D problem. We treat the case where ramifications are present in the one-dimensional domain and show some results on the geometry of an electrical treeing, a ramified structure that propagates in insulators causing their failure.
	\end{abstract}
	%
		%
	
	\section{Introduction}
	\label{section:introduction}

    The aim of this work is to obtain a geometrically reduced formulation of the electrostatic equation on two coupled domains representing materials with different electrical properties, more specifically, a thin fracture inside a wide three-dimensional domain. In particular, we are interested in modelling the electric field and potential inside the electrical treeing~\cite{bahadoorsingh2007role, buccella2023computational}, which is a self-propagating defect, characterized by long and thin branches~\cite{schurch2014imaging, schurch20193d}, causing the deterioration of insulating components of electrical cables. The defect is filled with gas, with a dielectric constant close to 1, while the external material is typically a solid insulator with higher dielectric constant.
    
    The inclusion of thin and ramified domains within wide three-dimensional volumes is a challenge common to many fields, such as the modeling of fluid flow in fractured porous media~\cite{lesinigo2011multiscale}, microvascular blood flow~\cite{formaggia2010cardiovascular} and drug delivery through microcirculation~\cite{cattaneo2014computational}, besides defect propagation inside dielectric materials. The discretization of problems on such domains involves high computational complexity, due to difficulties in mesh generation and a large number of degrees of freedom. One common approach to reduce this complexity consists in the approximation of the intricate inner domain as a one-dimensional domain, thereby reducing it to its skeleton~\cite{d2008coupling, formaggia2001coupling, cerroni2019mathematical}. This approximation allows to overcome the difficulty of generating a fine three-dimensional mesh on the inner thin domain and to consider a coarser mesh on the external domain as well.~\textcite{notaro2016mixed} proposed a mixed Finite Element Method (FEM) for the solution of coupled problems in 3D and 1D domains, extending the idea of~\textcite{d2012finite} for three-dimensional problems involving line sources.
    The presence of a line source still represents a challenge due to the singularity of the solution, as discussed in \cite{gjerde2020singularity}, where a singularity removal method is proposed.
    For a more accurate numerical solution around the 1D domain on coarser meshes, Extended Finite Elements (XFEM) can be applied, as done in~\textcite{bvrezina2021extended}, ~\textcite{gracie2010modelling} and ~\cite{grappein2024extended} for 3D problems with 1D source, and coupled mixed-dimensional coupled problems. 

    In this work we will reduce the treeing domain to a one-dimensional graph, adapt the full model introduced by~\textcite{VILLA2017687} to a mixed-dimensional framework and numerically solve it with the Finite Element Methods. The reduction of this problem presents some criticalities related to the profile of the electric potential in the gas domain, with a non-negligible dependence on the radial coordinate, and the presence of a jump in the normal component of the electric field across the interface between the two materials. We propose a dual-primal formulation of the problem, modelling the evolution of both electric field and potential in the solid 3D domain and only the potential in the 1D domain, where the electric field can be computed a posteriori, keeping into account also its components non-tangential to the 1D domain. To account for the dependence of the potential in the gas on the radial coordinate we rely on the knowledge of the potential profile associated with a constant, given charge distribution.

    Let us review the structure of the paper. In Section~\ref{section:the_3d_problem} we present the model equation and describe the equidimensional domains in the simple case where the inner one does not present branches. We then derive the mixed-dimensional formulation in Section~\ref{section:electric_field} and establish its well-posedness in Section~\ref{section:reduced}. We extend the model to also account for branches and bifurcations in Section~\ref{section:bifurcations} and introduce the numerical methods for the solution of the complete problem in Section~\ref{section:NumericalMethods}. Finally, we present three test cases in Section~\ref{section:Results}: the validation of the proposed reduction is on a simple geometry with a single one-dimensional line as inner domain, the application of the model to a short line immersed in a cylindrical 3D domain, and finally to the intricate structure of an electrical treeing.
  
	\section{The 3D problem}
	\label{section:the_3d_problem}
	To model the evolution of the electric field $\mathbf{E}$ and potential $\Phi$ on two domains filled with materials with different permittivities we consider the electrostatic equation \cite{griffiths2023introduction}. The problem is defined on a three-dimensional domain $\Omega$, such as the one represented in Figure~\ref{figure:3d-3d:domain}, composed of a subdomain $\Omega_g$, assumed cylindrical, typically filled with gas, surrounded by a solid insulator $\Omega_s$. We call $\Lambda$ the centerline of $\Omega_g$, defined as $\Omega_g = \{\mathbf{x}\ : \ \text{dist}(\mathbf{x},\Lambda) \leq R\}$. This inner domain represents a branch of the electrical treeing, which will be thoroughly treated in the next sections. We suppose that in $\Omega_s$ there is null electric charge, while in the gas the charge density is given by a function $q\ : \ \Omega_g \to \mathbb{R}$. Let us introduce the coefficient $\epsilon:\ \Omega\to\mathbb{R}$, modeling the dielectric constants taking values $\epsilon_s$ and $\epsilon_g$ in the two domains:
	
	\begin{equation*}
		\epsilon = 
		\begin{cases}
			\epsilon_s, &\text{in\ }\Omega_s, \\
			\epsilon_g, &\text{in\ }\Omega_g,
		\end{cases}
	\end{equation*}
	
	\noindent and consider as unknowns of our problem the displacement field $\mathbf{D}=\epsilon\mathbf{E}$ and the potential $\Phi$. We call $\mathbf{D}_s$, $\Phi_s$ and $\mathbf{D}_g$, $\Phi_g$ their restrictions to $\Omega_s$ and $\Omega_g$, respectively. Then, the system of equations we will consider is the following:
	
	\begin{equation}
		\begin{cases}
			\nabla \cdot \mathbf{D}_g = \dfrac{q}{\epsilon_0}, & \text{in\ } \Omega_g, \\
			\nabla \cdot \mathbf{D}_s = 0, & \text{in\ } \Omega_s, \\
			\mathbf{D} = -\epsilon\nabla\Phi, & \text{in\ } \Omega.
		\end{cases}
		\label{eq:3d-3d}
	\end{equation}
	
	\noindent We complete the problem with Neumann and Dirichlet boundary conditions on portions $\partial\Omega_{D}$ and $\partial\Omega_N$ of the boundary $\partial\Omega$, such that $\partial\Omega_D\cap\partial\Omega_N=\emptyset$ and $\partial\Omega_D\cup\partial\Omega_N=\partial\Omega$:
	
	\begin{equation}
		\begin{cases}
			\mathbf{D}\cdot\mathbf{n} = \nu, & \text{on\ } \partial\Omega_N, \\
			\Phi = \Phi^b, & \text{on\ } \partial\Omega_D,
		\end{cases}
		\label{eq:3d-3d:bc}
	\end{equation}
	
	\noindent where $\nu\ :\ \partial\Omega_D \to \mathbb{R} $ and $ \Phi^b : \ \partial\Omega_N\to\mathbb{R} $ are the known Neumann and Dirichlet terms, respectively. The Neumann condition sets the value of the normal component of the displacement field $\mathbf{D}$ on the boundary $\partial\Omega_N$, while the Dirichlet condition fixes the value of the potential $\Phi$ on the boundary $\partial\Omega_D$.
 
    Finally, we impose interface conditions on the surface $\Sigma = \{\mathbf{x}\ : \ \text{dist}(\mathbf{x},\Lambda) = R\}$ separating the two domains. In particular, we consider a simplification of the system proposed in \cite{villa2021discretization} where the displacement field presents a jump in the normal component, proportional to the total surface charge $q_\Gamma$ and the potential is continuous:
	
	\begin{subnumcases}{}
        \mathbf{D}_s\cdot\mathbf{n}_s + \mathbf{D}_g\cdot\mathbf{n}_g = -\dfrac{e}{\epsilon_0} q_\Gamma, & $\text{on\ } \Sigma,$
          \label{eq:3d-3d:ic:1} \\
        \Phi_s = \Phi_g, & $\text{on \ } \Sigma.$
          \label{eq:3d-3d:ic:2}
    \end{subnumcases}
	
	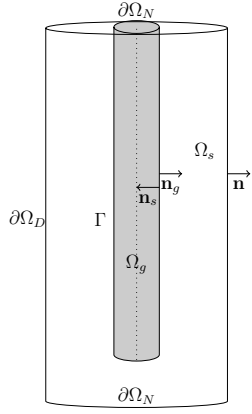
\begin{figure}
		\centering
		\resizebox{.2\textwidth}{!}
		{\begin{tikzpicture}[scale=0.8, every node/.style={scale=0.8}]
				\node (b) [cylinder, shape border rotate=90, draw, minimum height=7.5cm, minimum width=1cm, fill=black!20, line width = .5,yshift = .5cm] {};
				\node (a) [cylinder, shape border rotate=90, draw, minimum height=8.5cm, minimum width=4cm,line width = .5] {};
				\draw [->] (0.5,1) -- (1,1) node[midway,below] {$\mathbf{n}_g$};
				\node [yshift=-1cm] {$\Omega_g$};
				\node [xshift=1.5cm, yshift=1.5cm] {$\Omega_s$};
				\node [yshift=4.6cm] {$\partial\Omega_N$};
				\node [yshift=-3.9cm] {$\partial\Omega_N$};
				\node [xshift=-2.4cm] {$\partial\Omega_D$};
				\node [xshift=-.8cm] {$\Gamma$};
				\draw [->] (2,1) -- (2.5,1) node[midway,below] {$\mathbf{n}$};
				\draw [->] (0.5,0.7) -- (0,0.7) node[midway,below] {$\mathbf{n}_s$};
				\draw[dotted] (0,-3.1) -- (0,4.2);
		\end{tikzpicture}}
		\caption{Domain $\Omega$, given by a cylinder corresponding to the gas subdomain $\Omega_g$, surrounded by a generic volume $\Omega_s$, representing the dielectric domain. }
		\label{figure:3d-3d:domain}
	\end{figure}
	
	\noindent For the sake of simplicity, we start by considering two coaxial cylindrical domains $\Omega_g$ and $\Omega_s$, as in Figure~\ref{figure:3d-3d:domain}. We introduce a parametrization on the centerline $\Lambda$ of $\Omega_g$, so that we can define the coordinate $s\in[0,S]$ along it. In the following, we will assume that, if one endpoint of $\Lambda$ belongs to the external boundary $\partial\Omega$ and the other one is internal to the solid domain $\Omega_s$, the first one coincides with the point of coordinate $s=0$ and the second with $s=S$. For every point $\mathbf{x}_s\in\Lambda$, having coordinate $s\in[0,S]$, we define the transversal section of $\Omega_g$, orthogonal to the centerline, as $\mathcal{D}(s) = \{\mathbf{x}\in\Omega_g\ : \ |\mathbf{x} - \mathbf{x}_s| \leq R \ \text{and} \ (\mathbf{x}-\mathbf{x}_s)\perp \Lambda \}$. In particular, we will denote the basis of the cylinder $\Omega_g$ immersed in the solid domain by $\mathcal{D}(S)$, while the one belonging to the external boundary of $\Omega$ by $\mathcal{D}(0)$. Moreover, we call $\Gamma = \bigcup_{s\in[0,S]}\partial\mathcal{D}(s)$ the lateral surface of $\Omega_g$, so that the separating interface between the two domains in Figure~\ref{figure:3d-3d:domain} is $\Sigma = \Gamma\cup\mathcal{D}(S)$.
 
	If we assume that the inner cylinder is very thin, i.e. its radius in much smaller than its length, we can approximate the coupled problem described above as a mixed-dimensional one, where $\Omega_g$ is collapsed on its one-dimensional centerline $\Lambda$ and $\Omega_s$ is identified with the whole three-dimensional domain $\Omega$. Another simplification we introduce is the assumption that the charge $q$ is constant over sections of $\Omega_g$, orthogonal to the axis of the cylinder, more precisely:
 
	\begin{assumption}
		The gas domain is a cylinder with radius $R$ and length $L$, with $R\ll L$.
		\label{assumption:thin}
	\end{assumption}
	
	\begin{assumption}
		The total charge $q$ in $\Omega$ is constant over sections $\mathcal{D}(s)$ orthogonal to its centerline $\Lambda$, $\forall s\in\Lambda$.
		\label{assumption:constant_q}
	\end{assumption}
	
	\noindent In order to perform this dimensional reduction we start by separating the problems on the two domains, considering as unknowns the restrictions of the electric field and potentials on $\Omega_s$ and $\Omega_g$. At this stage the boundaries becomes the unions of the external boundaries of each domain and the separating surface: 
	
	\begin{equation*}
		\begin{aligned}
			\partial\Omega_s & = (\bar{\Omega}_s\cap \partial\Omega) \cup \Sigma;\\
			\partial\Omega_g & = (\bar{\Omega}_g\cap \partial\Omega) \cup \Sigma.
		\end{aligned}
	\end{equation*}
	
	\noindent The interface conditions on $\Sigma$ become boundary conditions  for the two problems in this framework. 
	
	
	\noindent We can separate the problems on the two domains, coupled by the interface conditions~\eqref{eq:3d-3d:ic:1}-\eqref{eq:3d-3d:ic:2}, and express the problem in the gas domain in primal form:
	
	Find $\mathbf{D}_s\ : \ \Omega_s\to\mathbb{R}^3,\ \Phi_s\ : \ \Omega_s\to\mathbb{R},\ \Phi_g\ : \ \Omega_g\to\mathbb{R}$ such that
	
	\begin{center}
		\begin{minipage}{.48\textwidth}
			\begin{subnumcases}{}
				\mathbf{D}_s + \epsilon_s\nabla \Phi_s = 0, &$  \text{in}\ \Omega_s, $
				\label{eq:P1:dual:1} \\
				\nabla\cdot\mathbf{D}_s = 0, & $ \text{in}\ \Omega_s, $
				\label{eq:P1:dual:2} \\
				\Phi_s = \Phi_g, & $ \text{on}\ \Sigma, $
				\label{eq:P1:dual:3} \\
				\mathbf{D}_s \cdot \mathbf{n}_s = -\mathbf{D}_g \cdot\mathbf{n}_g + g, & $ \text{on}\ \Sigma, $
				\label{eq:P1:dual:4}\\
				\Phi_s = \bar{\Phi}_s, & $ \text{on}\ \partial\Omega_{s,D},  $
				\label{eq:P1:dual:5}\\
				\mathbf{D}_s\cdot\mathbf{n} = \nu_s, & $ \text{on}\ \partial\Omega_{s,N}, $
				\label{eq:P1:dual:6}
			\end{subnumcases}
		\end{minipage}
		\hspace{.5em}
		\begin{minipage}{.48\textwidth}
			\begin{subnumcases}{}
				-\nabla\cdot\left(\epsilon_g\nabla\Phi_g\right) = f, & $ \text{in}\ \Omega_g $,
				\label{eq:P2:primal:1}\\
				\Phi_g = \Phi_s, & $ \text{on}\ \Sigma $,
				\label{eq:P2:primal:2} \\
				\epsilon_g \nabla\Phi_g \cdot \mathbf{n}_s = \mathbf{D}_s \cdot\mathbf{n}_g - g, & $ \text{on}\ \Sigma $,
				\label{eq:P2:primal:3}\\
				\Phi_g = \bar{\Phi}_g, & $ \text{on}\ \partial\Omega_{g,D} $,
				\label{eq:P2:primal:4}\\
				\nabla\Phi_g\cdot\mathbf{n} = \nu_g, & $ \text{on}\ \partial\Omega_{g,N}, $
				\label{eq:P2:primal:5}
			\end{subnumcases}
		\end{minipage}
	\end{center}
	\noindent where $f=\dfrac{q}{\epsilon_0}$, $g = -\dfrac{e}{\epsilon_0}q_\Gamma$, $\bar{\Phi}_s = \Phi\big|_{\partial\Omega_{s,D}}$, $\bar{\Phi}_g = \Phi\big|_{\partial\Omega_{g,D}}$, $\nu_s = \nu\big|_{\partial\Omega_{s,N}}$ and $\nu_g = -\nu\big|_{\partial\Omega_{g,N}}$.\\
    In the following, we will perform the dimensionality reduction by integrating the first equation~\eqref{eq:P2:primal:1} by parts, where the Neumann boundary conditions naturally appear, following the approach of \textcite{cerroni2019mathematical}. This implies that the electric field in the gas domain is not directly computed as an unknown of our problem, but only as a postprocessing.
 
	\section{Model reduction}
	\label{section:electric_field}
	In this section we will derive the reduced mixed-dimensional model for the electric field and potential on two domains, modeled as coaxial cylinders, taking into account the interface conditions that prescribe continuity of the potentials and a jump discontinuity on the normal component of the displacement fields across $\Gamma$. This kind of geometrical reduction is typical of coupled problems describing flow models, with very similar sets of equations as~\eqref{eq:3d-3d}. However, these models generally involve continuity of the scalar unknwon at the interface, as in equation~\eqref{eq:3d-3d:ic:2}, but do not present a jump of the normal component of the vectorial unknown, unlike what we have in equation~\eqref{eq:3d-3d:ic:1}  \cite{grappein2024extended}~\cite{martin2005modeling}. As we will see in Section~\ref{section:reduction_g}, another difference with respect to these works consists in the definition of the 1D variable, which is usually modeled as a constant, whereas in our case we are considering a splitting of $\Phi_g$ that allows us to model its radial variation.

	\subsection{Assumption on the potential}
	\label{section:dual_primal_formulation}
	We start by performing the reduction to one dimension of $\Omega_g$ and adapt the formulation of the problem in primal form in the gas domain (equations~\eqref{eq:P2:primal:1}-\eqref{eq:P2:primal:5}).
    Therefore, we want to end up with an unknown electric potential which is only dependent on the coordinate $s$. However, the hypothesis of a constant potential on each section $\mathcal{D}(s)$, made in the 1D reduction of problems in porous media, such as ~\cite{cerroni2019mathematical}~\cite{grappein2024extended}~\cite{martin2005modeling}, is restrictive because, together with potential continuity at the interface, it would imply that the electric field has only one component, tangent to $\Lambda$. Moreover, the right-hand-side of equation~\eqref{eq:P2:primal:1} represents the volume charge concentration, which was be assumed to be constant on sections (Assumption~\ref{assumption:constant_q}), and produces a non-negligible transversal electric field. As a consequence, the electric potential must be non-constant on sections.\\
	The potential produced by a constant concentration of charge on each section can be analytically computed, thanks to Gauss theorem. We integrate the divergence of the electric field and charge concentration over a cylinder $\mathcal{W}\subseteq\Omega_g$, of radius $r\leq R$ and height $h$, considering $h$ much smaller than the total length of $\Lambda$:   
	\begin{equation}
		\int_{\mathcal{W}} \nabla\cdot\left(\mathbf{E}_g\right) =\int_{\mathcal{W}}\frac{q}{\epsilon_0}.
		\label{gauss}
	\end{equation}
    \noindent Thanks to Assumptions~\ref{assumption:thin} and~\ref{assumption:constant_q}, and since, under the hypothesis of radially symmetric domain, the potential generated by uniform charge distribution has radial symmetry with respect to the center of the cylinder, the electric field line are orthogonal to the lateral surface $\mathcal{S}$ of $\mathcal{W}$. Then, the left hand side of equation~\eqref{gauss} can be rewritten as:
	\begin{equation*}
		\int_{\mathcal{W}} \nabla\cdot\left(\mathbf{E}_g\right) =\int_{\mathcal{S}}\mathbf{E}_g\cdot\mathbf{n}_g =
		-2\pi r \int_{s_1}^{s_2}\mathbf{E}_g(r,s)\cdot\mathbf{n}_g ds.
	\end{equation*}
	
	\noindent Then, equation~\eqref{gauss} is equivalent to
	\begin{equation}\nonumber
		-2\pi r \int_{s_1}^{s_2}\epsilon_g\nabla\Phi_g(r,s)\cdot\mathbf{n}_g ds = \pi r^2 \int_{s_1}^{s_2} \frac{q(s)}{\epsilon_0} ds, \quad \forall r\in(0,R],\ \forall s_1,s_2\in\Lambda,
	\end{equation}
	\noindent and therefore
	\begin{equation}
		\frac{\partial\Phi_g}{\partial r}(r,s) = -\frac{q(s)}{2\epsilon_0\epsilon_g}r, \quad \forall r\in(0,R], \ \forall s\in\Lambda.
	\end{equation}
	
	\noindent We finally integrate with respect to $r$ and obtain the analytic expression of the potential in the gas domain:
	\begin{equation}
		\Phi_g(r,s)-\Phi_g(0,s) = -\int_0^r \frac{q(s)}{2\epsilon_0\epsilon_g}\rho d\rho = -\frac{q(s)}{2\epsilon_0\epsilon_g} \frac{r^2}{2} = -\frac{q(s)}{4\epsilon_0\epsilon_g} r^2, \quad \forall r\in(0,R],\ \forall s\in\Lambda.
	\end{equation}
	
	\noindent Observe that the difference of potential on the left-hand side is continuous in the variable $r$ and tends to vanish as $r$ becomes smaller:
	
	\begin{equation*}
		\lim_{r\rightarrow 0}(\Phi_g(r,s)-\Phi_g(0,s))=0.
	\end{equation*}
	
	\noindent Let us now define the following functions:
	
	\begin{subequations}
		\begin{alignat}{6}
			\Phi_\Lambda \ & : \ \Lambda & \to \mathbb{R}, \quad & \Phi_\Lambda(s) & =&  \Phi_g(0,s),\quad & \forall s\in\Lambda,\\
			\Phi_r \ & : \ \Lambda & \to \mathbb{R}, \quad & \Phi_r(s) & =& -\dfrac{q(s)}{4\epsilon_0\epsilon_g},\quad & \forall s\in\Lambda,
			\label{def:phi_r}\\
			\phi \ & : \ [0,R] & \to \mathbb{R}, \quad & \phi(r) & =& r^2,\quad & \forall r\in[0,R].
			\label{def:phi}
		\end{alignat}
	\end{subequations}
	
	\noindent Here, $\Phi_r$ and $\phi$ are known terms in the expression of $\Phi_g$ and take into account the radial effect of a constant charge distribution on sections, while the dependence on the longitudinal coordinate $s$ is taken into account by the additive term $\Phi_\Lambda$, which is constant over sections:
 
	\begin{equation}
		\Phi_g(s,r) = \Phi_\Lambda(s) + \Phi_r(s)\phi(r),\quad\forall r\in[0,R(s)], \ \forall s\in\Lambda,
		\label{eq:splitting_phi_g}
	\end{equation}
	
	\noindent By substituting this splitting of the potential in the interface conditions~\eqref{eq:P2:primal:2}-~\eqref{eq:P2:primal:3}, we can now rewrite continuity of the potential and jump of the electric field on the lateral surface $\Gamma$ as follows:
	
	\begin{equation}
		\Phi_r = \frac{\Phi_s - \Phi_\Lambda}{\phi(R)},\qquad \text{on}\ \Gamma,
		\label{continuity_phi}
	\end{equation}
	\begin{equation*}
		\mathbf{D}_s\cdot\mathbf{n}_s = g + \epsilon_g\nabla\Phi_g\cdot\mathbf{n}_g = g + \epsilon_g \nabla(\Phi_\Lambda(s) + \Phi_r(s)\phi(r))\cdot \mathbf{n}_g = g +\epsilon_g \Phi_r\phi'(R), \quad \text{on\ }\Gamma,
	\end{equation*}
	
	\noindent and, combining them, we obtain a Robin interface condition:
	
	\begin{equation}
		\mathbf{D}_s\cdot\mathbf{n}_s = g + \epsilon_g \frac{\Phi_s - \Phi_\Lambda}{\phi(R)}\phi'(R), \qquad \text{on}\ \Gamma.
		\label{robin_phi}
	\end{equation}
	
	\noindent Observe that condition~\eqref{continuity_phi} implies that $\Phi_s$ on $\Gamma$ can only depend on the coordinate $s$, which means that the trace $\hat{\Phi}_s$ of $\Phi_s$ on $\Gamma$ can be approximated as a constant on the boundary of each section $\mathcal{D}$ by its integral mean. 

    \begin{equation}
		\hat{\Phi}_s(s) := \frac{1}{|\partial \mathcal{D}(s)|}\int_{\partial \mathcal{D}(s)}\Phi_s.
        \label{eq:phi_mean}
	\end{equation}
    Moreover, since $\Omega_g$ is thin (Assumption~\ref{assumption:thin}), if necessary we will extend $\Phi_s$ inside $\Omega_g$ as a constant.

	\subsection{Reduction of the equation in the gas}
	\label{section:reduction_g}
	
	\noindent If we substitute equation~\eqref{eq:splitting_phi_g} in equation~\eqref{eq:P2:primal:1}, we obtain:
	
	\begin{equation}
		\begin{split}
			\nabla\cdot\left( \epsilon_g\nabla\Phi_g \right) &=
			\dfrac{\epsilon_g}{r}\dfrac{\partial}{\partial r}\left( r\dfrac{\partial\Phi_g}{\partial r} \right) + \frac{\partial}{\partial s} \left(\epsilon_g \dfrac{\partial\Phi_g}{\partial s}\right) = \\
			&= \epsilon_g \left(\frac{1}{r}\frac{\partial \Phi_g}{\partial r} + \frac{\partial^2 \Phi_g}{\partial r^2}\right) + \frac{d \epsilon_g}{d s}\frac{\partial\Phi_g}{\partial s} + \epsilon_g\frac{\partial^2 \Phi_g}{\partial s^2}= \\
			&= \epsilon_g \left(\frac{1}{r}\frac{\partial \Phi_g}{\partial r} + \frac{\partial^2 \Phi_g}{\partial r^2} + \dfrac{\partial^2\Phi_g}{\partial s^2} \right) + \frac{d\epsilon_g}{d s}\frac{\partial\Phi_g}{\partial s}.
		\end{split}
		\label{eq:gas:1}
	\end{equation}
	
	\noindent Since the only dependence of $\Phi_g$ on $r$ is due to $\phi$, its partial derivatives with respect to this coordinate are simply:
	
	\begin{equation*}
			\dfrac{\partial \Phi_g}{\partial r} = \Phi_r \dfrac{d\phi}{d r} = 2r\Phi_r, \quad
			\dfrac{\partial^2 \Phi_g}{\partial r^2} = \Phi_r \dfrac{d^2\phi}{d r^2} = 2\Phi_r.
	\end{equation*}
	
	\noindent The dependence on $s$, instead is due to $\Phi_\Lambda$ and $\Phi_r$ and the corresponding partial derivatives are given by:
	
	\begin{equation}
        \dfrac{\partial\Phi_g}{\partial s} = \dfrac{\partial \Phi_\Lambda}{\partial s} + \dfrac{\partial\Phi_r}{\partial s}\phi, \quad
        \dfrac{\partial^2\Phi_g}{\partial s^2} = \dfrac{d^2 \Phi_\Lambda}{d s^2} + \dfrac{d^2\Phi_r}{d s^2}\phi.
        \label{eq:dphig_ds}
    \end{equation}
	
	\noindent We can substitute these expression into equation~\eqref{eq:gas:1} and obtain:
	
	\begin{multline}
		\nabla\cdot(\epsilon_g\nabla\Phi_g)
		= \epsilon_g\left( 4\Phi_r + \frac{d^2\Phi_\Lambda}{ds^2} + r^2\frac{d^2\Phi_r}{ds^2}\right) + \frac{d\epsilon_g}{d s}\left(\frac{d\Phi_\Lambda}{d s} + r^2\frac{d\Phi_r}{d s} \right), \quad \forall r<R, \ \forall s\in\Lambda,
		\label{eq:gas:2}
	\end{multline}
	\noindent which implies
	
	\begin{equation*}
		\epsilon_g\left( 4\Phi_r + \frac{d^2\Phi_\Lambda}{ds^2} + r^2\frac{d^2\Phi_r}{ds^2}\right) + \frac{d\epsilon_g}{d s}\left(\frac{d\Phi_\Lambda}{d s} + r^2\frac{d\Phi_r}{d s} \right) = -\dfrac{q(s)}{\epsilon_0}, \quad \forall r<R, \ \forall s\in\Lambda.
	\end{equation*}
	
	\noindent Note that this equation is not pointwise satisfied, since the left-hand side depends on $r$ but the right-hand side does not, as a consequence of Assumption~\ref{assumption:constant_q}, neglecting the radial profile of the charge concentration $q$; in fact this equation is only satisfied in mean, over sections $\mathcal{D}(s), \ \forall s\in\Lambda$. Thus, we start by integrating equation~\eqref{eq:P2:primal:1}, combined with equation~\eqref{eq:gas:2} over a section $\mathcal{D}$:

	\begin{equation*}
		\begin{split}
			&\int_{\mathcal{D}(s)} \nabla\cdot\left( \epsilon_g(s)\nabla \Phi_g \right) = \\
			&= \int_{\mathcal{D}(s)} \left[\epsilon_g(s)\left( 4\Phi_r(s) + \frac{d^2\Phi_\Lambda}{ds^2}(s) + r^2\frac{d^2\Phi_r}{ds^2}(s)\right) + \frac{d\epsilon_g}{d s}(s)\left(\frac{d\Phi_\Lambda}{d s}(s) + r^2\frac{d\Phi_r}{d s}(s) \right) \right]
		\end{split}
	\end{equation*}
	
	\noindent Since the first integrand is constant over sections, while the second one depends on $r$, this equation becomes:
	
	\begin{equation}
		\begin{split}
			\int_{\mathcal{D}(s)} \nabla\cdot\left( \epsilon_g(s)\nabla \Phi_g \right) &= \\
			&\begin{split}
				= |\mathcal{D}(s)| &\left[ \epsilon_g(s)\left( 4\Phi_r(s) + \frac{d^2\Phi_\Lambda}{ds^2}(s)\right) + \frac{d\epsilon_g}{d s}(s)\frac{d\Phi_\Lambda}{d s}(s) \right] +\\
				&+ \epsilon_g(s) \frac{d^2\Phi_r}{d s^2}(s) \int_{\mathcal{D}(s)}r^2 + \frac{d\epsilon_g}{d s}(s) \frac{d\Phi_r}{d s}(s) \int_{\mathcal{D}(s)}r^2 =
			\end{split} \\
			&\begin{split}
				= \pi R^2 &\left[ \epsilon_g(s)\left( 4\Phi_r(s) + \frac{d^2\Phi_\Lambda}{ds^2}(s)\right) + \frac{d\epsilon_g}{d s}(s)\frac{d\Phi_\Lambda}{d s}(s) \right] +\\
				&+  \pi \frac{R^4}{2} \left(\epsilon_g(s) \frac{d^2\Phi_r}{d s^2}(s) + \frac{d\epsilon_g}{d s}(s) \frac{d\Phi_r}{d s}(s)\right), \quad\forall s \in \Lambda.
			\end{split}
		\end{split}
        \label{eq:1d:derivation}
	\end{equation}
	
	\noindent The continuity condition~\eqref{continuity_phi} allows us to replace $\Phi_r$ in the first term with the difference between $\Phi_\Lambda$ and $\Phi_s$ on the interface, where $\Phi_s$ is equivalent to its integral mean $\hat{\Phi}_s$, according to equation~\eqref{eq:phi_mean}:
	
	\begin{multline*}
		\pi R^2 \left[ \epsilon_g(s)\left( 4\Phi_r(s) + \frac{d^2\Phi_\Lambda}{ds^2}(s)\right) + \frac{d\epsilon_g}{d s}(s)\frac{d\Phi_\Lambda}{d s}(s) \right] = \\
		= \pi R^2\left[ \epsilon_g(s)\left( 4\frac{\hat{\Phi}_s - \Phi_\Lambda(s)}{R^2} + \frac{d^2\Phi_\Lambda}{ds^2}(s)\right) + \frac{d\epsilon_g}{d s}(s)\frac{d\Phi_\Lambda}{d s}(s) \right].
	\end{multline*}
	
	\noindent We integrate now also the right-hand side of equation~\eqref{eq:P2:primal:1} over a section, recalling that the source term is given by the total charge $f = \frac{q(s)}{\epsilon_0}$ on $\mathcal{D}(s)$, $\forall s\in\Lambda$:
	
	\begin{equation}\nonumber
		\int_{\mathcal{D}(s)} f = \frac{q(s)}{\epsilon_0} |\mathcal{D}(s)| = \frac{q(s)}{\epsilon_0} \pi R^2, \qquad \forall s \in \Lambda.
	\end{equation}
	
	\noindent Moreover, subsituting $\dfrac{d}{ds}\left(\epsilon_g\dfrac{d\phi_\Lambda}{ds}\right) = \epsilon_g\dfrac{d^2\Phi_\Lambda}{ds^2} + \dfrac{d\epsilon_g}{ds}\dfrac{d\Phi_\Lambda}{ds}$ and $\dfrac{d}{ds}\left(\epsilon_g\dfrac{d\Phi_r}{ds}\right)= \epsilon_g\dfrac{d^2\Phi_r}{ds^2} + \dfrac{d\epsilon_g}{ds}\dfrac{d\Phi_r}{ds}$ in equation~\eqref{eq:1d:derivation}, we obtain the final formulation of the one-dimensional equation in the gas domain:
	
	\begin{equation*}
		-\pi R^2 \frac{d}{ds}
		\left(\epsilon_g(s)\frac{d\Phi_\Lambda}{d s}(s)\right) + 4\pi \epsilon_g(s) \left(\Phi_\Lambda(s) - \hat{\Phi}_s(s)\right)
		= \pi R^2 \frac{q(s)}{\epsilon_0} + \dfrac{\pi R^4}{2} \dfrac{d}{ds}\left( \epsilon_g\dfrac{d\Phi_r}{ds} \right), \quad \forall s \in \Lambda.
	\end{equation*}
	
	\noindent Finally, since we are considering $R\to 0$, we can drop the higher order term $\dfrac{\pi R^4}{2} \dfrac{d}{ds}\left( \epsilon_g\dfrac{d\Phi_r}{ds} \right)$ on the right-hand side of the equation and get:
	
	\begin{equation}
		-\pi R^2 \frac{d}{ds}
		\left(\epsilon_g(s)\frac{d\Phi_\Lambda}{d s}(s)\right) + 4\pi \epsilon_g(s) \left(\Phi_\Lambda(s) - \hat{\Phi}_s(s)\right) = \pi R^2 \frac{q(s)}{\epsilon_0}, \qquad \forall s \in \Lambda.
		\label{eq:gas:3}
	\end{equation}
	
	\noindent The one-dimensional domain coincides with the centerline $\Lambda$ of $\Omega_g$ and the boundary conditions on it must be imposed only on the two endpoints $s=0$ and $s=S$. The interface conditions on $\Gamma$ were incorporated through the previous calculations into the governing equation and resulted in a coupling reaction term. We are only left to reduce to one dimension the Dirichlet condition on $\mathcal{D}(0)$ and the interface conditions on $\mathcal{D}(S)$.
	
	On $\mathcal{D}(0)$ we can rewrite the boundary condition introducing the splitting of $\Phi_g$ from equation~\eqref{eq:splitting_phi_g}: 
	
	\begin{equation*}
		\bar{\Phi}_g = \Phi_g\big|_{\mathcal{D}(0)} = \Phi_\Lambda(0) + \Phi_r(0)\phi(r) = \Phi_\Lambda(0) + \Phi_r(0) r^2,
		\quad r\in[0,R].
	\end{equation*}
	
	\noindent We can integrate over $\mathcal{D}(0)$ and obtain
	
	\begin{equation}\nonumber
		\int_{\mathcal{D}(0)} \bar{\Phi}_g =
		\pi R^2 \Phi_\Lambda(0) - 2\pi \frac{R^4}{16\epsilon_0\epsilon_g}q(0).
	\end{equation}
	
	\noindent If we denote by $\hat{\bar{\Phi}}_g$ the integral of $\bar{\Phi}_g$ over $\mathcal{D}(0)$:
	\begin{equation}\nonumber
		\hat{\bar{\Phi}}_g := \int_{\mathcal{D}(0)} \bar{\Phi}_g,
	\end{equation}
	
	\noindent we can write the boundary condition on $s=0$ as follows:
	
	\begin{equation*}
		\Phi_\Lambda(0) = \hat{\bar{\Phi}}_g - 2\pi \frac{R^4}{16\epsilon_0\epsilon_g}q(0).
	\end{equation*}
	
	\noindent Observe that $\Phi_\Lambda(0) \rightarrow \hat{\bar{\Phi}}_g, \ \text{as\ } R\rightarrow 0$.
	Then, the Dirichlet boundary condition on $s=0$ can be approximated as:
	
	\begin{equation}
		\Phi_\Lambda(0) = \hat{\bar{\Phi}}_g.
		\label{eq:gas:dirichlet}
	\end{equation}
	
	\noindent Let us now construct a Neumann boundary condition at $s=S$ by integrating the jump condition given by equation~\eqref{eq:P2:primal:3} on $\mathcal{D}(S)$:
	
	\begin{equation}\nonumber
		\int_{\mathcal{D}(S)}\mathbf{D}_s\cdot\mathbf{s} =
		\int_{\mathcal{D}(S)}g + \int_{\mathcal{D}(S)}\epsilon_g\nabla\Phi_g\cdot\mathbf{s} =
		|\mathcal{D}(S)| g(S) + \int_{\mathcal{D}(S)}\epsilon_g\nabla\Phi_g\cdot\mathbf{s}.
	\end{equation}
	
	\noindent We can substitute $\mathbf{D}_s$ with $-\epsilon_s\nabla\Phi_s$, and $\nabla\Phi_g\cdot\mathbf{s}$ with its expression~\eqref{eq:dphig_ds}, and the previous equation becomes:
	
	\begin{equation*}
		\begin{split}
			-\int_{\mathcal{D}(S)}\epsilon_s\frac{\partial\Phi_s}{\partial s} &=
			|\mathcal{D}(S)|g(S) + |\mathcal{D}(S)| \epsilon_g(S) \dfrac{d\Phi_\Lambda}{ds}(S) +\epsilon_g(S) \dfrac{d\Phi_r}{ds}(S)\int_{\mathcal{D}(S)} \phi(r) = \\
            & = \pi R^2 g(S) + \pi R^2 \epsilon_g(S) \dfrac{d\Phi_\Lambda}{ds}(S) + \dfrac{\pi R^4}{4} \epsilon_g(S) \dfrac{d\Phi_r}{ds}(S)
		\end{split}
	\end{equation*}
	
	\noindent Thanks to Assumption~\ref{assumption:thin} of thin gas domain, we can approximate the potential $\Phi_s$ as a constant on the whole section, as in equation~\eqref{eq:phi_mean} and neglect the higher order term $\dfrac{\pi R^4}{4} \epsilon_g(S) \dfrac{d\Phi_r}{ds}(S)$:
	
	\begin{equation}
		-\epsilon_s\frac{d\hat{\Phi}_s}{ds}(S) =  g(S) + \epsilon_g \frac{d\Phi_\Lambda}{ds}(S).
        \label{eq:gas:neumann:phis}
	\end{equation}
	
	\noindent Collecting~\eqref{eq:gas:3},~\eqref{eq:gas:dirichlet} and~\eqref{eq:gas:neumann:phis}, we obtain the 1D version of problem~\eqref{eq:P2:primal:1}-~\eqref{eq:P2:primal:4}:
	
	\begin{equation}
		\begin{dcases}
			-\pi R^2 \dfrac{d}{ds}\left(\epsilon_g\dfrac{d\Phi_\Lambda}{d s}\right) + 4\pi \epsilon_g \left(\Phi_\Lambda - \hat{\Phi}_s\right) 
			= \pi R^2 \dfrac{q}{\epsilon_0}, \qquad \text{on}\ \Lambda,\\
			\Phi_\Lambda(0) = \hat{\bar{\Phi}}_g, \\
			\epsilon_g \frac{d\Phi_\Lambda}{ds}(S) = -\epsilon_s\frac{d\hat{\Phi}_s}{ds}(S) -g(S).
		\end{dcases}
		\label{eq:P2:primal:1d}
	\end{equation}
	
	\subsection{Reduction of the problem in the dielectric}
	\label{section:reduction_s}
	
	Consider now the problem~\eqref{eq:P1:dual:1}-~\eqref{eq:P1:dual:6} in the dielectric domain $\Omega_s$ in dual form, with the Robin interface condition~\eqref{robin_phi}. Since in the previous section we have reduced the gas domain to a one-dimensional line, in the final formulation of the problem obtained in this section we will extend the dielectric domain $\Omega_s$ to the whole domain $\Omega$. Let us now define the functional spaces to which $\mathbf{D}_s$ and $\Phi_s$ belong as $H(\text{div};\Omega)$ and $W_s:=L^2(\Omega_s)$, respectively. Notice that, however, $\nabla\Phi_s = -\epsilon_s^{-1}\mathbf{D}_s\in L^2(\Omega_s)$ and therefore $\Phi_s\in H^1(\Omega_s)$. We substitute equation~\eqref{eq:P1:dual:1} into equation~\eqref{eq:P1:dual:2} and multiply it by a test function belonging to the same space of $\Phi_s$, $w_s\in H^1(\Omega_s)$, and integrate by parts over $\Omega_s$:
	
	\begin{multline}\nonumber
		0 = \int_{\Omega_s} \nabla\cdot\mathbf{D}_s w_s = - \int_{\Omega_s} \mathbf{D}_s \cdot \nabla w_s + \int_{\partial\Omega_{s,D}} w_s \mathbf{D}_s \cdot \mathbf{n} +  \int_{\partial\Omega_{s,N}} w_s \mathbf{D}_s \cdot \mathbf{n} + \\ +  \int_{\Gamma} w_s \mathbf{D}_s \cdot \mathbf{n}_s + \int_{\mathcal{D}(S)} w_s \mathbf{D}_s \cdot \mathbf{s}, \qquad \forall w_s\in H^1(\Omega_s).
	\end{multline}
	
	\noindent Imposing the Neumann boundary condition on $\partial\Omega_{s,N}$ and Assumption~\ref{assumption:null_flux}, we obtain the following expression:
	
	\begin{equation}
		- \int_{\Omega_s} \mathbf{D}_s \cdot \nabla w_s + \int_{\Gamma} w_s \mathbf{D}_s \cdot \mathbf{n}_s = - \int_{\partial\Omega_{s,N}} w_s \nu, \qquad \forall w_s \in H^1_{0,\partial\Omega_{s,D}}(\Omega_s), 
		\label{eq:P1':weak:1}
	\end{equation}
	
	\noindent where $H^1_{0,\partial\Omega_{s,D}}(\Omega_s) := \{w\in H^1(\Omega_s)\ :\ w=0 \text{\ on \ } \partial\Omega_{s,D}\}$. Now we can apply the Robin condition~\eqref{robin_phi} and obtain:
	
	\begin{multline}
		\int_\Gamma w_s \mathbf{D}_s\cdot\mathbf{n}_s = \int_\Gamma w_s \left(g + \epsilon_g\frac{\Phi_s-\Phi_\Lambda}{\phi(R)}\phi'(R)\right) = \\
		= \int_\Lambda g \int_ {\partial\mathcal{D}}w_s  + \int_\Lambda \epsilon_g \frac{\phi'(R)}{\phi(R)}\int_{\partial\mathcal{D}}w_s \Phi_s - \int_\Lambda \epsilon_g \frac{\phi'(R)}{\phi(R)}\Phi_\Lambda \int_{\partial\mathcal{D}}w_s.
		\label{eq:P1':weak:2}
	\end{multline}
	
	\noindent Assume that the electric potential $\Phi_s$ and the test functions $w_s\in H^1_{0,\partial\Omega_{s,D}}(\Omega_s)$ on $\partial\mathcal{D}(s), \ s\in\Lambda,$ can be written as the sum of their integral mean over $\partial\mathcal{D}(s)$, which is constant on $\partial\mathcal{D}(s)$ and approximately equal to the respective mean over $\mathcal{D}(s)$, and a non-constant fluctuation around it:
	
	\begin{equation*}
		\Phi_s = \hat{\Phi}_s + \tilde{\Phi}_s; \qquad w_s = \hat{w}_s + \tilde{w}_s, \qquad \text{on\ } \Gamma,
	\end{equation*}
	
	\noindent and make the following assumption, similar to the one made by \textcite{cerroni2019mathematical}, on the fluctuations:
	
	\begin{assumption}
		The fluctuations of functions in $W_s$ around their integral mean on $\partial\mathcal{D}(s)$ have zero mean for all $s\in\Lambda$, and so does the product of the fluctuations of two different functions in $W_s$, i.e.
		\begin{equation}\nonumber
			\forall v=\hat{v}+\tilde{v},w=\hat{w}+\tilde{w}\in W_s, \qquad \int_{\partial\mathcal{D}} \tilde{w} \approx 0 \quad \text{and} \quad \int_{\partial\mathcal{D}} \tilde{v} \tilde{w} \approx 0.
		\end{equation}
		\label{assumption:fluctuations}
	\end{assumption}
	
	\noindent As a consequence, we can rewrite the integrals over $\partial\mathcal{D}$ in equation~\eqref{eq:P1':weak:2} as:
	\begin{equation}\nonumber
		\int_{\partial\mathcal{D}} w_s = \int_{\partial\mathcal{D}} \hat{w}_s + \int_{\partial\mathcal{D}} \tilde{w}_s \approx \int_{\partial\mathcal{D}} \hat{w}_s = |\partial\mathcal{D}| \hat{w}_s,
	\end{equation}
	\noindent and
	\begin{equation}\nonumber
		\int_{\partial\mathcal{D}} \Phi_s w_s = \int_{\partial\mathcal{D}} \hat{w}_s \hat{\Phi}_s + \int_{\partial\mathcal{D}} \tilde{w}_s \tilde{\Phi}_s + \hat{w}_s \int_{\partial\mathcal{D}} \tilde{\Phi}_s + \hat{\Phi}_s \int_{\partial\mathcal{D}} \tilde{w}_s \approx \int_{\partial\mathcal{D}} \hat{w}_s\hat{\Phi} = |\partial\mathcal{D}| \hat{w}_s\hat{\Phi}_s.
	\end{equation}
	
	\noindent We can substitute these integrals in~\eqref{eq:P1':weak:2} and back into equation~\eqref{eq:P1':weak:1} and obtain the following weak equation:
	
	\begin{equation}\nonumber
		-\int_{\Omega_s} \epsilon_s\mathbf{D}_s\cdot\nabla w_s + \int_\Lambda \epsilon_g\frac{\phi'(R)}{\phi(R)} \hat{w}_s |\partial\mathcal{D}|\left(\hat{\Phi}_s - \Phi_\Lambda\right) = -\int_\Lambda g \hat{w}_s |\partial\mathcal{D}|,  \qquad \forall w_s \in H^1_{0,\partial\Omega_{s,D}}(\Omega_s).
	\end{equation}
	
	\noindent Substituting now $\phi$ and its derivative with their analytical expressions and $|\partial\mathcal{D}|=2\pi R$, we obtain the final weak formulation of the primal problem in the dielectric domain:
	
	\begin{equation}
		\int_{\Omega} \epsilon_s\mathbf{D}_s\cdot\nabla w_s + 4\pi \epsilon_g\int_\Lambda \hat{w}_s\left(\Phi_\Lambda - \hat{\Phi}_s \right) = 2\pi R\int_\Lambda g \hat{w}_s, \qquad \forall w_s \in H^1_{0,\partial\Omega_{s,D}}(\Omega_s).
		\label{problem:weak:eq2}
	\end{equation}
	
	\noindent In order to go back to the strong problem and to the dual mixed formulation, we can integrate back by parts~\eqref{problem:weak:eq2} and obtain:
	
	\begin{equation}\nonumber
		\int_\Omega \left(-\nabla\cdot\mathbf{D}_sw_s + 4\pi  \epsilon_g(\Phi_\Lambda-\hat{\Phi}_s)\hat{w}_s\delta_\Lambda \right) = \int_\Omega 2\pi R g \hat{w}_s \delta_\Lambda, \qquad \forall w_s = \hat{w}_s +\tilde{w}_s \in H^1_{0,\partial\Omega_{s,D}}(\Omega_s).
	\end{equation}
	
	\noindent In particular, this holds for $w_s = \hat{w}_s$.\\
	We have obtained a strong equation in the whole domain, with a line source term on $\Lambda$:
	
	\begin{equation}\nonumber
		\nabla\cdot\mathbf{D}_s - 4\pi  \epsilon_g(\Phi_\Lambda-\hat{\Phi}_s)\delta_\Lambda = -2\pi R g  \delta_\Lambda, \qquad \text{in}\ \Omega.
	\end{equation}
	
	\noindent If we simplify the coefficients of this equation and substitute $\mathbf{D}_s\cdot\mathbf{n}_s = -\epsilon_s\nabla\Phi_s$, we retrieve the strong dual mixed formulation of the problem in the dielectric domain with a line source concentrated on $\Lambda$ and a coupling term with the 1D problem~\eqref{eq:P2:primal:1d}:
	
	\begin{equation}
		\begin{dcases}
			\epsilon_s^{-1} \mathbf{D}_s + \nabla\Phi_s = 0, & \text{in}\ \Omega, \\
			\nabla\cdot \mathbf{D}_s - 4\pi  \epsilon_g(\Phi_\Lambda-\hat{\Phi}_s)\delta_\Lambda = -2\pi Rg\delta_\Lambda, &\text{in}\ \Omega, \\
			\Phi_s = \bar{\Phi}_s, & \text{on}\ \partial\Omega_{D}, \\
			\mathbf{D}_s\cdot\mathbf{n} = \nu, & \text{on}\ \partial\Omega_N.
		\end{dcases}
		\label{P1':1d}
	\end{equation}
	
	\section{Reduced 3D-1D coupled problem}
	\label{section:reduced}
	
	\noindent The final mixed-dimensional dual-primal coupled problem is the following:
	
		\begin{subnumcases}{}
			\epsilon_s^{-1} \mathbf{D}_s  + \nabla\Phi_s = 0, & $\text{in}\ \Omega, $
		\label{eq:3d-1d:1}\\
			\nabla\cdot \mathbf{D}_s - 4\pi \epsilon_g(\Phi_\Lambda - \hat{\Phi}_s)\delta_\Lambda = -2\pi Rg\delta_\Lambda, & $\text{in}\ \Omega, $
		\label{eq:3d-1d:2}\\
				-\pi R^2 \dfrac{d}{ds}\left(\epsilon_g\dfrac{d\Phi_\Lambda}{d s}\right) + 4\pi \epsilon_g \left(\Phi_\Lambda - \hat{\Phi}_s\right) 
				= \pi R^2 \dfrac{q}{\epsilon_0} 
			& $\text{on}\ \Lambda, $
		\label{eq:3d-1d:3}\\
			\Phi_\Lambda(0) = \hat{\bar{\Phi}}_g, 
		\label{eq:3d-1d:4}\\
			\epsilon_g\dfrac{d\Phi_\Lambda}{ds} (S) =
			- g(S) - \epsilon_s\dfrac{d\hat{\Phi}_s}{ds}(S), 
		\label{eq:3d-1d:5}\\
			\Phi_s = \bar{\Phi}_s, & $\text{on}\ \partial\Omega_{D},$
		\label{eq:3d-1d:6}\\
			\mathbf{D}_s\cdot\mathbf{n} = \nu, & $\text{on}\ \partial\Omega_N.$
		\label{eq:3d-1d:7}
		\end{subnumcases}

    \noindent Note that, despite the differences in the derivation, similar coupling terms between problems in mixed-dimensional domains can also be found in the context of fluid flow, \cite{notaro2016mixed} \cite{gjerde2019splitting}.

    We can observe that not only is a coupling between the problems in the two domains present in the equations~\eqref{eq:3d-1d:2} and~\eqref{eq:3d-1d:3}, but also it appears at the tip of the one-dimensional reduced domain as a Neumann condition, similar to the jump interface condition~\eqref{eq:3d-3d:ic:1}. However, the value of the coefficient $4\pi\epsilon_g$ makes the coupling term in equation~\eqref{eq:3d-1d:3} predominant in the evolution of $\Phi_g$ and the coupling at the tip of $\Lambda$ negligible. Moreover, since the area $\mathcal{D}(S)$ is small, according to Assumption~\ref{assumption:thin}, we can assume that the flux exchange between the gas and the dielectric domain happens mostly through the lateral surface $\Gamma$ of $\Omega_g$ and the contribution across $\mathcal{D}(S)$ is negligible:
        
    \begin{assumption}
        The flux of the electric field across $\mathcal{D}(S)$ is negligible, i.e. $\int_{\mathcal{D}(S)} \nabla\Phi_s\cdot\mathbf{s} \approx 0.$ 
        \label{assumption:null_flux}
    \end{assumption}
    
    \noindent This way we obtain an alternative boundary condition for the 1D problem, which is simply:
    
    \begin{equation}
        \frac{d\Phi_\Lambda}{ds}(S) = 
        -\frac{1}{\epsilon_g} g(S).
        \label{eq:gas:neumann}
    \end{equation}
    
    \begin{remark}
        Alternatively, we could also choose to substitute the coupling term in equations~\eqref{eq:3d-1d:2} and~\eqref{eq:3d-1d:3} by the known quantity $\Phi_r\phi(R)$, as in~\eqref{continuity_phi}. In this case, we would not be allowed to rely on Assumption~\ref{assumption:null_flux} without actually decoupling the two problems, and we would need keep condition~\eqref{eq:3d-1d:5} as it is, introducing a weak coupling at the tip of the 1D domain. Moreover, in this case we would not be imposing the condition~\eqref{continuity_phi} of continuity of the potentials across the interface $\Gamma$ between the two original domains, and should take it into account as one additional equation.
    \end{remark}

    \begin{remark}
        Under the Assumptions~\ref{assumption:thin},~\ref{assumption:constant_q},~\ref{assumption:fluctuations}, the effect of a uniform volume charge distribution in a cylinder and of a line charge distribution on its centerline are equivalent outside of $\Omega_g$. Indeed, the flux of the displacement field $\mathbf{D}_s$ on a cylindrical surface $\partial \mathcal{W}$ surrounding $\Omega_g$ remains the same in the two cases. 
        We can compute it by applying the Gauss theorem to the original 3D-3D problem, taking into account Assumption~\ref{assumption:constant_q} of constant charge over sections of $\Omega_g$:
        \begin{equation}
            \int_{\partial\mathcal{W}}\mathbf{D}_s\cdot\mathbf{n} = \int_{\mathcal{W}}\dfrac{q}{\epsilon_0} = \int_\Lambda \dfrac{q}{\epsilon_0} |\mathcal{B}|, 
            \label{eq:remark:3d-3d}
        \end{equation}
        \noindent where $\mathcal{B}$ denotes a transversal section of $\mathcal{W}$. If we do the same on the reduced 3D-1D problem, we obtain:
        \begin{equation*}
            \int_{\partial\mathcal{W}} \mathbf{D}_s\cdot\mathbf{n} = \int_\mathcal{W}\dfrac{q}{\epsilon_0}\delta_\Lambda = \int_\Lambda \dfrac{q}{\epsilon_0} |\mathcal{B}|,
        \end{equation*}
        \noindent which is equivalent to equation~\eqref{eq:remark:3d-3d}.
    \end{remark}
	
	\subsection{Well-posedness of the dual-primal problem}
	In this section we will prove that the coupled 3D-1D problem~\eqref{eq:3d-1d:1}-~\eqref{eq:3d-1d:7} admits a unique weak solution.\\
	A similar problem, in primal form, was proven to be well-posed by \textcite{dangelo2008coupling} on weighted Sobolev spaces. As we will see in the following section, we do not need to define a continuous trace operator and, consequently, to work with such spaces.\\
	The well-posedness of a similar problem in mixed dimensions was also studied by \textcite{bvrezina2021extended}, where however the one-dimensional equation is not expressed in primal form, as we have in equation~\eqref{eq:3d-1d:3}. Mixed-dimensional problems in dual-primal form can be found in~\cite{formaggia2018analysis} and~\cite{antonietti2016mimetic}, but defined on subdomains of codimension 1.
 
	We start by gathering the weak formulation of problem~\eqref{eq:3d-1d:1}-\eqref{eq:3d-1d:7}, imposing the Neumann boundary condition on $\partial\Omega_N$ with the Lagrange multiplier $\xi=-\mathbf{D}_s\cdot\mathbf{n}\big|_{\partial\Omega_N}$. Let us define the following spaces:

    \begin{equation*}
        H^{1/2}(\partial\Omega_N) = \{ \phi \big|_{\partial\Omega_N} \ : \ \phi\in H^1(\Omega) \};
    \end{equation*}
	
	\begin{equation*}
		W_\Lambda = H^1(\Lambda) = \left\{ \psi \in L^2(\Lambda) \ :\ \dfrac{d\psi}{ds} \in L^2(\Lambda)\right\};
	\end{equation*}
	
	\begin{equation*}
        W_s = L^2(\Omega);
	\end{equation*}
 
	\begin{equation*}
		V_s = \{ \mathbf{v}\in H(\text{div};\Omega) \ :\ \mathbf{v}\cdot\mathbf{n} \in L^2(\partial\Omega) \},
	\end{equation*}
	
	\noindent where $V_s$ is a subspace of $H(\text{div};\Omega) = \{\mathbf{v}\in L^2(\Omega) \ : \ \nabla\cdot\mathbf{v}\in L^2(\Omega) \}$ that takes into account the boundary conditions on $\partial\Omega_N$. Note that we have required extra regularity on the trace of $\mathbf{v}\cdot\mathbf{n}$ on the boundary.\\
    On these spaces we consider the norms $\|\lambda\|_{H^{1/2}(\partial\Omega_N)} = \inf\{\|\phi\|_{H^1(\Omega)}\ : \ \phi\in H^1(\Omega), \ \phi\big|_{\partial\Omega_N} = \lambda \}$, $\|\psi\|_{W_\Lambda^0} = \|\psi\|_{L^2(\Lambda)}$, $\|\phi\|_{W_s} = \|\phi\|_{H^1(\Omega)}$ and $\|\mathbf{v}\|_{V_s}^2 = \|\nabla\cdot\mathbf{v}\|_{L^2(\Omega)}^2 + \|\mathbf{v}\|_{L^2(\Omega)}^2 + \|\mathbf{v}\cdot\mathbf{n}\|_{L^2(\partial\Omega)}^2$.\\
	
	If we integrate by parts equations~\eqref{eq:3d-1d:1}-\eqref{eq:3d-1d:3}, substitute the boundary conditions~\eqref{eq:3d-1d:4}-\eqref{eq:3d-1d:7} and exploit Assumption~\ref{assumption:fluctuations}, as in  the derivation of equation~\eqref{problem:weak:eq2}, we obtain the following problem:\\
	
	Find $\left( \left(\mathbf{D}_s, \xi\right), \Phi_s,\Phi_\Lambda \right)\in \left(V_s \times H^{1/2}(\partial\Omega_N)\right) \times W_s\times W_\Lambda$ such that:
	
	\begin{subnumcases}
		\mathcal{A}(\mathbf{D}_s,\mathbf{v}) +\mathcal{B}(\mathbf{v},\left(\Phi_s,\xi)\right)  = -<\bar{\Phi}_s,\mathbf{v}\cdot\mathbf{n}>_{\partial\Omega_D}, & $ \forall \mathbf{v}\in V_s,  $
		\label{eq:3d-1d:weak:1}\\
		\mathcal{B}(\mathbf{v},\left(\phi,\lambda)\right)  + c_{\Lambda s}(\Phi_\Lambda,\hat{\phi}) - c_{ss}(\hat{\Phi}_s, \hat{\phi}) = <G, \hat{\phi}>_\Lambda - <\nu,\phi>_{\partial\Omega_N}, & $ \forall \phi\in W_s, \forall \lambda\in H^{1/2}(\partial\Omega_N), $
		\label{eq:3d-1d:weak:2} \\
		\mathcal{A}_\Lambda (\Phi_\Lambda, \psi) + c_{\Lambda \Lambda}(\Phi_\Lambda,\psi) -c_{\Lambda s}(\hat{\Phi}_s, \psi) = <F,\psi>_\Lambda, & $ \forall \psi\in W_\Lambda^0, $
		\label{eq:3d-1d:weak:3}
	\end{subnumcases} 
	
	\noindent where we have defined the following operators:
	
	\begin{equation*}
		\begin{alignedat}{5}
			&\mathcal{A}\ :& \ V_s\times V_s &\longrightarrow \mathbb{R}, \qquad& \mathcal{A}(\mathbf{u},\mathbf{v}) &= \epsilon_s^{-1}\int_\Omega \mathbf{u}\cdot\mathbf{v}; \\
			&\mathcal{B}\ :& \ V_s\times \left(W_s\times H^{1/2}(\partial\Omega_N)\right) &\longrightarrow \mathbb{R}, \qquad& \mathcal{B}\left(\mathbf{u},(\phi,\lambda)\right) &= - \int_\Omega \phi \nabla\cdot\mathbf{u} + \int_{\partial\Omega_N} \lambda \mathbf{v}\cdot\mathbf{n}; \\
			&c_{\Lambda\Lambda}\ :&\ W_\Lambda\times W_\Lambda &\longrightarrow \mathbb{R}, \qquad& c_{\Lambda\Lambda}(\psi_1,\psi_2) &= 4\pi\epsilon_g\int_\Lambda \psi_1\psi_2; \\
			&c_{\Lambda s}\ :&\ W_\Lambda\times W_s &\longrightarrow \mathbb{R}, \qquad& c_{\Lambda,s}(\psi,\phi) &= 4\pi\epsilon_g\int_\Lambda \psi\hat{\phi}; \\
			&c_{ss}\ :&\ W_s\times W_s &\longrightarrow \mathbb{R}, \qquad& c_{ss}(\phi_1,\phi_2) &= 4\pi\epsilon_g\int_\Lambda \hat{\phi_1}\hat{\phi_2}; \\
			&\mathcal{A}_\Lambda \ :&\ W_\Lambda\times W_\Lambda &\longrightarrow \mathbb{R}, \qquad& \mathcal{A}_\Lambda(\phi,\psi) &= \pi R^2 \int_\Lambda\epsilon_g\dfrac{d\psi}{ds}\dfrac{d\phi}{ds};\\
			&G\ :& \ \Lambda &\longrightarrow \mathbb{R}, \qquad & G &= \dfrac{g}{2\epsilon_g}R; \\
			&F\ :& \ \Lambda & \longrightarrow \mathbb{R}, \qquad & F &= \pi R^2 \dfrac{q}{4\epsilon_0}.
		\end{alignedat}
	\end{equation*}
	
	\noindent We assume that $g\in L^2(\Lambda)$, $q\in L^2(\Lambda)$, $\bar{\Phi}_s\in H^{-1/2}(\partial\Omega_D)$, $\nu\in H^{-1/2}(\partial\Omega_N)$, $\epsilon_g\in L^\infty(\Lambda)$, $\epsilon_s\in L^\infty(\Omega)$ and $\epsilon_s\geq\epsilon_g\geq 1$.\\
	If we sum equation~\eqref{eq:3d-1d:weak:3} and equation~\eqref{eq:3d-1d:weak:1}, we obtain the following equivalent problem:\\
	
	Find $\left( \left(\mathbf{D}_s, \Phi_\Lambda\right), \left(\Phi_s,\xi\right) \right)\in \left(\left(V_s \times W_\Lambda\right) \times \left(W_s\times H^{1/2}(\partial\Omega_N)\right)\right)$ such that:
	
	\begin{equation}
		\begin{cases}
            \begin{split}
    			a\left(\left(\mathbf{D}_s,\Phi_\Lambda\right),\left(\mathbf{v},\psi\right)\right) + b_1\left(\left(\mathbf{v},\psi\right),\left(\Phi_s,\xi\right)\right) = -<\bar{\Phi}_s,\mathbf{v}\cdot\mathbf{n}>_{\partial\Omega_D} + <F,\psi>_\Lambda , \\ \forall \left(\mathbf{v},\psi\right)\in V_s\times W_\Lambda, 
            \end{split}
            \\
            \begin{split}
                b_2\left(\left(\mathbf{D_s},\Phi_\Lambda\right),\left(\phi,\lambda\right)\right) - c_{ss}\left( \Phi_s,\phi \right)= <G, \hat{\phi}>_\Lambda - <\nu,\phi>_{\partial\Omega_N}, 
                \\ \forall (\phi,\lambda)\in W_s\times H^{1/2}(\partial\Omega_N).
            \end{split}
		\end{cases}
		\label{eq:3d-1d:weak:eq}
	\end{equation} 
	
	\noindent where we define a norm over the product spaces $V_s\times W_\Lambda$ and $W_s\times H^{1/2}(\partial\Omega_N)$ respectively as:
	
	\begin{equation*}
        \begin{aligned}
            \|\left( \mathbf{v}, \psi \right)\|_{V_s\times W_\Lambda} &= \sqrt{\|\mathbf{v} \|_{V_2}^2 + \|\psi\|_{W_\Lambda}^2},\\
		      \|\left(\phi,\xi\right)\|_{W_s\times H^{1/2}(\partial\Omega_N)} &= \sqrt{\|\phi\|_{W_s}^2 + \|\xi\|_{H^{1/2}(\partial\Omega_N)}^2}
        \end{aligned}
	\end{equation*}
	
	\noindent and the operators $a\ :\ \left( V_s\times W_\Lambda \right)\times\left( V_s\times W_\Lambda \right) \to \mathbb{R}$, $b_1 \ :\ \left( V_s\times W_\Lambda\right) \times \left(W_s\times H^{1/2}(\partial\Omega_N)\right) \to \mathbb{R}$ and $b_2 \ :\ \left( V_s\times W_\Lambda\right) \times \left(W_s\times H^{1/2}(\partial\Omega_N)\right) \to \mathbb{R}$ respectively as follows:
	
	\begin{equation*}
		\begin{aligned}
		    a\left( \left( \mathbf{v}_1,\psi_1 \right), \left( \mathbf{v}_2,\psi_2 \right) \right) &= \mathcal{A}_s(\mathbf{v}_1,\mathbf{v}_2) + \mathcal{A}_\Lambda(\psi_1,\psi_2) + c_{\Lambda\Lambda}(\psi_1,\psi_2), \\
            b_1\left( \left(\mathbf{v},\psi\right), \left( \phi,\xi \right) \right) & = \mathcal{B}(\mathbf{v},(\phi,\xi)) - c_{\Lambda s}(\psi,\phi), \\
            b_2\left( \left(\mathbf{v},\psi\right), \left( \phi,\xi \right) \right) & = \mathcal{B}(\mathbf{v},(\phi,\xi)) + c_{\Lambda s}(\psi,\phi).
		\end{aligned}
	\end{equation*}
	
	\noindent We can prove that this saddle point problem admits a unique solution, relying on the results by \textcite{nicolaides1982existence}, \textcite{bernardi1988generalized}, \textcite{brezzi2012mixed}.
	
	In order to show this result, we need to prove that the integral mean of a function in $L^2(\Omega)$ belongs to $L^2(\Lambda)$, and we do it in a similar way as \textcite{cerroni2019mathematical}.
	
	\begin{lemma}
		If $\phi\in L^2(\Omega)$, then $\hat{\phi}\in L^2(\Lambda)$ and $\exists k>0$ such that the following inequality holds:
		\begin{equation*}
			\|\hat{\phi}\|^2_{L^2(\Lambda)}\leq k\|\phi\|^2_{L^2(\Omega)}.
		\end{equation*}
		\label{lemma:phi_hat_L2}
	\end{lemma}
	
	\begin{proof}
		By definition of the $L^2$-space, $\hat{\phi}\in L^2(\Lambda)\iff\int_\Lambda \hat{\phi}^2 ds < \infty$. Then, we consider
		\begin{equation}
			\int_\Lambda \hat{\phi}^2 ds = \int_\Lambda \left(\frac{1}{|\mathcal{D}(s)|}\int_{\mathcal{D}(s)} \phi \right)^2.
			\label{eq:proof:phi_hat}
		\end{equation}
		\noindent By Jensen's inequality,
		\begin{equation*}
			\left(\frac{1}{|\mathcal{D}(s)|}\int_{\mathcal{D}(s)} \phi \right)^2 \leq \frac{1}{|\mathcal{D}(s)|^2}\int_{\mathcal{D}(s)} \phi^2,
		\end{equation*}
		\noindent the right-hand side of equation~\eqref{eq:proof:phi_hat} can be rewritten as follows:
		\begin{equation*}
			\int_\Lambda\left(\frac{1}{|\mathcal{D}(s)|}\int_{\mathcal{D}(s)} \phi \right)^2 \leq \int_\Lambda\frac{1}{|\mathcal{D}(s)|^2}\int_{\mathcal{D}(s)} \phi^2.
		\end{equation*}
		\noindent Observe that the double integral over $\Lambda\times\mathcal{D}(s)$ is equivalent to the integral of $\phi$ over $\Omega_g$ and, as we assume $|\mathcal{D}(s)|=\pi R^2,\ \forall s\in\Lambda$, then, we have
		\begin{equation*}
			\int_\Lambda \hat{\phi}^2 ds \leq \frac{1}{|\mathcal{D}(s)|^2}\int_{\Omega_g} \phi^2 = \frac{1}{(\pi R^2)^2} \int_{\Omega_g} \phi^2 \leq \frac{1}{(\pi R^2)^2} \int_{\Omega} \phi^2 = \frac{1} {(\pi R^2)^2}\|\phi\|^2_{L^2(\Omega)} < \infty,
		\end{equation*}
		\noindent since $\phi^2\geq 0$ and $\Omega_g\subset\Omega$.
		
		We have proved that $\hat{\phi}\in L^2(\Lambda)$ and $\|\hat{\phi}\|^2_{L^2(\Lambda)}\leq k\|\phi\|^2_{L^2(\Omega)}$, with $k=\dfrac{1}{(\pi R^2)^2}$.
	\end{proof}
	
	\noindent We also need to prove some results on the operators involved in the problem that will allow us to show the well-posedness of~\eqref{eq:3d-1d:1}-\eqref{eq:3d-1d:7}.
	
	\begin{lemma}
		$G$ and $F$ belong to $L^2(\Lambda)$ and the right-hand sides of the equations ~\eqref{eq:3d-1d:weak:eq} are continuous.
	\end{lemma}
	\begin{proof}
		A straightforward consequence of the assumption $g,q\in L^2(\Lambda)$ is that also $G$ and $F$ belong to $L^2(\Lambda)$.\\
        For the continuity, we start by applying the triangular inequality and Cauchy-Schwarz inequality:
		\begin{multline*}
				|<G,\hat{\phi}>_\Lambda - <\nu,\phi>_{\partial\Omega_N}|
                \leq |<G,\hat{\phi}>_\Lambda| + |<\nu,\phi>_{\partial\Omega_N}|
				\leq \\
				\leq \|G\|_{L^2(\Lambda)}\|\hat{\phi}\|_{L^2(\Lambda)} + \|\nu\|_{H^{-1/2}(\partial\Omega_N)} \|\phi\|_{H^{1/2}(\partial\Omega_N)}, \qquad
                \forall \phi\in W_s.			
		\end{multline*}
		\noindent By Lemma~\ref{lemma:phi_hat_L2},
		\begin{multline*}
				|<G,\hat{\phi}>_\Lambda+<\nu,\phi>_{\partial\Omega_N}| 
                \leq \max\left\{\dfrac{\|G\|_{L^2(\Lambda)}}{\pi R^2}, \|\nu\|_{H^{-1/2}(\partial\Omega_N)}\right\} \|\phi\|_{H^1(\Omega)} \leq \\
				\leq \max\left\{ \dfrac{\|G\|_{L^2(\Lambda)}}{\pi R^2}, 	\|\nu\|_{H^{-1/2}(\partial\Omega_N)}\right\}\|\left(\phi,\psi\right)\|_{W_s\times W_\Lambda}, \qquad \forall (\phi,\xi)\in W_s\times H^{1/2}(\partial\Omega_N).
		\end{multline*}

        \noindent Finally, by triangular inequality and Cauchy-Schwarz inequality,

        \begin{multline*}
            \left| -<\bar{\Phi}_s,\mathbf{v}\cdot\mathbf{n}>_{\partial\Omega_D} + <F,\psi>_\Lambda \right| \leq
            \|\bar{\Phi}_s\|_{L^2(\partial\Omega_D)}\|\mathbf{v}\cdot\mathbf{n}\|_{L^2(\partial\Omega_D)} + \|F\|_{L^2(\Lambda)}\|\psi\|_{L^2(\Lambda)} \leq\\
            \leq \|\bar{\Phi}_s\|_{L^2(\partial\Omega_D)}\|\mathbf{v}\|_{V_s} + \|F\|_{L^2(\Lambda)}\|\psi\|_{L^2(\Lambda)} \leq 
            \max\left\{ \|\bar{\Phi}_s\|_{L^2(\partial\Omega_D)},\|F\|_{L^2(\Lambda)} \right\} \|\mathbf{v},\psi\|_{V_s\times W_\Lambda}, 
            \\ \forall \left(\mathbf{v},\psi\right)\in V_s\times W_\Lambda.
        \end{multline*}
	\end{proof}
	
	\begin{lemma}
		The bilinear form $a$ is positive definite, i.e.
        \begin{equation*}
            a\left(\left(\mathbf{u}_1,\psi_1\right),\left(\mathbf{u}_2,\psi_2\right)\right)> 0,\ \forall\left(\mathbf{u}_1,\psi_1\right),\ \left(\mathbf{u}_2,\psi_2\right)\in V_s,
        \end{equation*}
        continuous and coercive on the kernel of $b_1$ and of $b_2$,
        \begin{equation}
            \ker(b_1) = \ker(b_2) = \ker(b)=\{(\mathbf{u},\psi)\in V_s\times W_s\ : \ \nabla\cdot \mathbf{u} = 0,\ \mathbf{u}\cdot\mathbf{n} = 0\ \text{on\ }\partial\Omega_N\}.
            \label{eq:kerB}
        \end{equation}
	\end{lemma}
	
	\begin{proof}
		For the continuity, by triangular and Cauchy-Schwarz inequality,
		\begin{equation*}
            \begin{split}
    			\left|a\left( \left(\mathbf{u}_1,\psi_1\right),\left(\mathbf{u}_2,\psi_2\right) \right)\right| 
                &= \left| \epsilon_s^{-1} \int_\Omega \mathbf{u}_1\cdot\mathbf{u}_2 + \int_\Lambda \pi R^2 \epsilon_g \dfrac{d\psi_1}{ds}\dfrac{d\psi_2}{ds} + 4\pi\epsilon_g\psi_1\psi_2 \right| \leq \\
                &\begin{split}
                    \leq (\min_{\Omega}(\epsilon_s))^{-1} \|\mathbf{u}_1\|_{L^2(\Omega)}\|\mathbf{u}_2\|_{L^2(\Omega)} &+ \pi R^2 \max_{\Lambda}(\epsilon_g)|\psi_1|_{H^1(\Omega)} |\psi_2|_{H^1(\Omega)} +\\
                    &+ 4\pi \max_\Lambda(\epsilon_g)\|\psi_1\|_{L^2(\Lambda)} \|\psi_2\|_{L^2(\Lambda)}
                \end{split}
            \end{split}
		\end{equation*}

        \noindent Moreover, exploiting the definitions of the norms in $V_s$ and $W_\Lambda$ and on their product space:

        \begin{equation*}
            \begin{split}
                \left|a\left( \left(\mathbf{u}_1,\psi_1\right),\left(\mathbf{u}_2,\psi_2\right) \right)\right| 
                &\leq K_1\left( \|\mathbf{u}_1\|_{L^2(\Omega)} \|\mathbf{u}_2\|_{L^2(\Omega)} + \|\psi_1\|_{H^1(\Lambda)}\|\psi_2\|_{H^1(\Lambda)} \right) \leq\\
                &\leq K_1 \|\left(\mathbf{u}_1,\psi_1\right)\|_{V_s\times W_\Lambda} \|\left(\mathbf{u}_2,\psi_2\right)\|_{V_s\times W_\Lambda}
                , \qquad \forall \left(\mathbf{u},\psi\right)\in V_s\times W_\Lambda,
            \end{split}
        \end{equation*}
        \noindent with $K_1>0$. 
		
		\noindent Moreover, by definition of the norms in $V_s$ and $W_\Lambda$ and on their product space and recalling that $\mathbf{u}\in\ker(b)$ implies $\nabla\cdot\mathbf{u}=0$, $\mathbf{u}\cdot\mathbf{n} = 0$, and thus $\|\mathbf{u}\|_{V_s} = \|\mathbf{u}\|_{L^2(\Omega)}$, $\forall \mathbf{u}\in\ker(b)$, we conclude that $a$ is coercive on $\ker(b)$:
  
		\begin{equation}
            \begin{split}
    			a\left(\left(\mathbf{u},\psi\right),\left(\mathbf{u},\psi\right)\right) &\geq \|\epsilon_s\|_{L^\infty(\Omega)}^{-1}\int_\Omega |\mathbf{u}|^2 + \int_\Lambda \pi R^2 \min_{\Lambda}(\epsilon_g) \left(\dfrac{d\psi}{ds}\right)^2 + \int_\Lambda 4\pi\min_{\Lambda}(\epsilon_g) \psi^2 \geq\\
                &\geq\|\epsilon_s\|_{L^\infty(\Omega)}^{-1}\|\mathbf{u}\|_{L^2(\Omega)}^2 + \pi R^2 \|\nabla\psi\|_{H^1(\Lambda)}^2 + 4 \pi \|\psi\|_{L^2(\Lambda)}^2 \geq\\
                &\geq K_2 \left(\|\left( \mathbf{u},\psi \right)\|_{V_s\times W_\Lambda}\right),
                \qquad \forall \left(\mathbf{u},\psi\right)\in \ker(b),
            \end{split}
            \label{eq:lemma:a:coercivity}
		\end{equation}	
        
        \noindent with $K_2>0$.\\
        Finally, $a$ is also positive definite, as a consequence of equation~\eqref{eq:lemma:a:coercivity}:
        \begin{equation*}
            \begin{split}
                a\left(\left(\mathbf{u},\psi\right),\left(\mathbf{u},\psi\right)\right) 
                \geq\|\epsilon_s\|_{L^\infty(\Omega)}^{-1}\|\mathbf{u}\|_{L^2(\Omega)}^2 + \pi R^2\min_\Lambda(\epsilon_g) \|\nabla\psi\|_{H^1(\Lambda)}^2 + 4 \pi \min_\Lambda(\epsilon_g) \|\psi\|_{L^2(\Lambda)}^2 \geq 0,& \\
                \forall \left(\mathbf{u},\psi\right)\in V_s\times W_\Lambda.&
            \end{split}
        \end{equation*}
        
	\end{proof}

    \begin{lemma}
        The operators $b_1$ and $b_2$ are continuous.
    \end{lemma}

    \begin{proof}
        By triangular and Cauchy-Schwarz inequality,
        \begin{equation}
            \begin{split}
                \left| b_1\left(\left(\mathbf{u},\psi\right),\left(\phi,\xi\right)\right) \right| & =
                \left| -\int_\Omega \phi\nabla\cdot \mathbf{u} + \int_{\partial\Omega_N}\xi\mathbf{u}\cdot\mathbf{n} - \int_\Lambda \psi\hat{\phi} \right| \leq\\
                &\leq \|\phi\|_{L^2(\Omega)} \|\nabla\cdot\mathbf{u}\|_{L^2(\Omega)} + \|\xi\|_{H^{1/2}(\partial\Omega_N)}\|\mathbf{u}\cdot\mathbf{n}\|_{L^2(\Omega)} + \|\psi\|_{L^2(\Omega)}\|\hat{\phi}\|_{L^2(\lambda)}.
            \end{split}
            \label{eq:lemma:inf-sup}
        \end{equation}

        \noindent and the same holds for $b_2$:
        \begin{equation}
            \begin{split}
                \left| b_2\left(\left(\mathbf{u},\psi\right),\left(\phi,\xi\right)\right) \right| & =
                \left| -\int_\Omega \phi\nabla\cdot \mathbf{u} + \int_{\partial\Omega_N}\xi\mathbf{u}\cdot\mathbf{n} + \int_\Lambda \psi\hat{\phi} \right| \leq\\
                &\leq \|\phi\|_{L^2(\Omega)} \|\nabla\cdot\mathbf{u}\|_{L^2(\Omega)} + \|\xi\|_{H^{1/2}(\partial\Omega_N)}\|\mathbf{u}\cdot\mathbf{n}\|_{L^2(\Omega)} + \|\psi\|_{L^2(\Omega)}\|\hat{\phi}\|_{L^2(\lambda)}.
            \end{split}
            \label{eq:lemma:inf-sup:b2}
        \end{equation}
        
        \noindent By exploiting the definition of the norms in $V_s$, $W_s$ and $W_\Lambda$ and Lemma~\ref{lemma:phi_hat_L2}, we can show that $b$ is continuous:
        \begin{equation*}
            \begin{split}
                \left| b_i\left(\left(\mathbf{u},\psi\right),\left(\phi,\xi\right)\right) \right| & \leq 
                \|\phi\|_{W_s} \|\mathbf{u}\|_{V_s} + \|\xi\|_{H^{1/2}(\partial\Omega_N)}\|\mathbf{u}\cdot\mathbf{n}\|_{V_s} +\|\psi\|_{W_\Lambda}\dfrac{\|\phi\|_{L^2(\Omega)}}{\pi R^2} \leq \\
                & \leq K_3 \|\left(\mathbf{u},\psi\right)\|_{V_s\times W_\Lambda}  \|\left(\phi,\xi\right)\|_{W_s\times H^{1/2}(\partial\Omega_N)},\quad i=1,2,
            \end{split}
        \end{equation*}
        \noindent with $K_3>0$.
    \end{proof}
	
	\begin{lemma}
		The bilinear forms $b_i,\ i=1,2,$ satisfy the $\inf-\sup$ condition:
        \begin{multline}
            \exists K_i>0 \text{\ such \ that}\ 
            \sup_
            {\scriptsize
                \begin{split}
                \left(\mathbf{v},\psi\right)&\in V_s\times W_\Lambda \\
                \mathbf{v} &\neq \mathbf{0} ,\ 
                \psi \neq 0
            \end{split}}
            \dfrac{b_i\left(\left(\mathbf{v},\psi\right),(\phi,\lambda)\right)}{\|\left(\mathbf{v},\psi\right)\|_{V_s\times W_\Lambda}}
            \geq K_i \|(\phi,\lambda)\|_{W_s\times H^{1/2}(\partial\Omega_N)}, \\ \forall (\phi,\lambda)\in W_s\times H^{1/2}(\partial\Omega_N), \ i=1,2.
            \label{eq:lemma:b:inf-sup}
        \end{multline}
	\end{lemma}
	
	\begin{proof}
        As in~\cite{babuvska2003mixed}, we first prove that $\exists C_1>0$ such that
        
        \begin{equation}
            \sup_
            {\scriptsize
                \begin{split}
                \left(\mathbf{v},\psi\right)&\in V_s\times W_\Lambda \\
                \mathbf{v} &\neq \mathbf{0},\
                \psi \neq 0
            \end{split}}
            \dfrac{b_i\left(\left(\mathbf{v},\psi\right),(\phi,\lambda)\right)}{\|\left(\mathbf{v},\psi\right)\|_{V_s\times W_\Lambda}}
            \geq C_1 \|\phi\|_{W_s}, \quad \forall \phi,\in W_s
            \label{eq:lemma:b:inf-sup:phi}
        \end{equation}

        \noindent and then that $\exists C_2 >0 $ such that

        \begin{equation}
            \sup_
            {\scriptsize
                \begin{split}
                \left(\mathbf{v},\psi\right)&\in V_s\times W_\Lambda \\
                \mathbf{v} &\neq \mathbf{0}, \
                \psi \neq 0
            \end{split}}
            \dfrac{b_i\left(\left(\mathbf{v},\psi\right),(\phi,\lambda)\right)}{\|\left(\mathbf{v},\psi\right)\|_{V_s\times W_\Lambda}}
            \geq K \|\mu\|_{H^{1/2}(\partial\Omega_N}, \quad \forall \mu \in H^{1/2}(\partial\Omega_N).
            \label{eq:lemma:b:inf-sup:mu}
        \end{equation}

        \noindent Finally, according to Theorem 3.1 of~\cite{gatica2008characterizing},~\eqref{eq:lemma:b:inf-sup:phi} and~\eqref{eq:lemma:b:inf-sup:mu} are necessary and sufficient conditions for~\eqref{eq:lemma:b:inf-sup}.
        
		  \noindent Let $(\phi,\lambda)\in W_s\times H^{1/2}(\partial\Omega_N)$ and $z\in H^1(\Omega)$ be the solution to
        \begin{equation*}
            \begin{dcases}
                -\Delta z = \phi, &\text{in\ } \Omega \\
                z = 0, &\text{on\ } \partial\Omega_D, \\
                \nabla z\cdot \mathbf{n} = 0, &\text{on\ } \partial \Omega_N,\ i=1,2.
            \end{dcases}
        \end{equation*}
        \noindent Fixed $\mathbf{\tilde{v}}\in V_s$ such that $\mathbf{\tilde{v}} = \nabla z$, then $-\nabla\cdot \mathbf{\tilde{v}} = \phi$, $\mathbf{\tilde{v}}\cdot\mathbf{n}=0$ on $\partial\Omega_N$.\\
        Moreover, $\|\mathbf{\tilde{v}}\|_{L^2(\Omega)} \leq \|z \|_{L^2(\Omega)} \leq C \|\phi \|_{L^2(\Omega)}$ for some $C>0$, by continuity of the solution to the Laplace problem. \\
        Let $w\in H^1(\Lambda)$ be the solution to

        \begin{equation}
            \begin{dcases}
                -\dfrac{d^2 w}{ds^2} &= \hat{\phi}, \quad \text{on\ }\Lambda, \\
                w(0) &= 0, \\
                \dfrac{dw}{ds}(S) & = 0.
            \end{dcases}
            \label{eq:lemma:b:inf-sup:phi:1}
        \end{equation}
        \noindent Fixed $\tilde{\psi}_i\in W_s,\ i=1,2,$ such that $\tilde{\psi}_2=w$ and $\tilde{\psi}_1=-w$, the continuous dependence of the solution to problem~\eqref{eq:lemma:b:inf-sup:phi:1} implies $ \|\tilde{\psi}_i\|_{H^1({\Lambda})} \leq C_i \|\hat{\phi}\|_{L^2(\Lambda)}$, for some $C_i>0$, $ i=1,2$. As a consequence, by Lemma~\ref{lemma:phi_hat_L2}, $\|\psi_1\|_{H^1(\Lambda)} = \|\psi_1\|_{W_\Lambda} \leq \dfrac{\|\phi\|_{L^2(\Omega)}}{\pi R^2}, \ i=1,2$. \\
        Then,

        \begin{equation}
            \begin{split}
                \sup_
                {\scriptsize
                    \begin{split}
                    \left(\mathbf{v},\psi\right)&\in V_s\times W_\Lambda \\
                    \mathbf{v} &\neq \mathbf{0} \\
                    \psi &\neq 0
                \end{split}}
                &\dfrac{b_i\left(\left(\mathbf{v},\psi\right),(\phi,\lambda)\right)}{\|(\mathbf{v},\psi)\|_{V_s\times W_\Lambda}}
                =
                \dfrac{b_i\left(\left(\mathbf{\tilde{v}},\tilde{\psi}_i\right),(\phi,\lambda)\right)}{\|\left(\mathbf{\tilde{v}},\tilde{\psi}_i\right)\|_{V_s\times W_\Lambda}} =\\
                &= \dfrac{1}{\left(\|\mathbf{\tilde{v}}\|_{V_s}^2 + \|\tilde{\psi}_i\|_{W_\Lambda}^2\right)^{1/2}}\left( -\int_\Omega \phi\nabla\cdot \mathbf{\tilde{v}} + \int_{\partial\Omega_N} \lambda\mathbf{\tilde{v}}\cdot\mathbf{n}  -\int_\Lambda\tilde{\psi}_i\dfrac{d^2\tilde{\psi}_i}{ds^2}\right)\geq\\
                &\geq \dfrac{1}{C\|\phi\|_{W_s}}\left( -\int_\Omega \phi^2  -\int_\Lambda\tilde{\psi}_i\dfrac{d^2\tilde{\psi}_i}{ds^2} \right), \quad \forall (\phi,\lambda)\in W_s\times H^{1/2}(\partial\Omega_N), \ i=1,2.
            \end{split}
            \label{eq:lemma:b:inf-sup:2}
        \end{equation}

        \noindent We can integrate the last term by parts and apply boundary conditions of problem~\eqref{eq:lemma:b:inf-sup:phi:1}:

        \begin{equation*}
            -\int_\Lambda\tilde{\psi}\dfrac{d^2\tilde{\psi}}{ds^2} = \int_\Lambda\left(\dfrac{d\tilde{\psi}}{ds}\right)^2 - \left(\tilde{\psi}(S)\dfrac{d\tilde{\psi}}{ds}(S) - \tilde{\psi}(0)\dfrac{d\tilde{\psi}}{ds}(0)\right) = \left\|\dfrac{d\tilde{\psi}}{ds}\right\|_{L^2(\Lambda)}^2.
        \end{equation*}

        \noindent If we substitute the equation above in equation~\eqref{eq:lemma:b:inf-sup:2}, we obtain~\eqref{eq:lemma:b:inf-sup:phi}:

        \begin{multline*}
            \sup_
                {\scriptsize
                    \begin{split}
                    \left(\mathbf{v},\psi\right)&\in V_s\times W_\Lambda \\
                    \mathbf{v} &\neq \mathbf{0} \\
                    \psi &\neq 0
                \end{split}}
                \dfrac{b_i\left(\left(\mathbf{v},\psi\right),(\phi,\lambda)\right)}{\|(\mathbf{v},\psi)\|_{V_s\times W_\Lambda}}
                \geq 
                \dfrac{1}{C\|\phi\|_{W_s}}\left(\|\phi\|_{W_s}^2 + \left\|\dfrac{d\tilde{\psi}_i}{ds}\right\|_{L^2(\Lambda)}^2\right)
                \geq C_1 \|\phi\|_{W_s}, \\ \forall (\phi,\lambda)\in W_s\times H^{1/2}(\partial\Omega_N), \ i=1,2.
        \end{multline*}

        \noindent Let now $\nu\in H^{-1/2}(\partial\Omega_N)$ and $t\in H^1(\Omega)$ be the solution to
        \begin{equation*}
            \begin{dcases}
                -\Delta t = 0, &\text{in\ } \Omega \\
                t = 0, &\text{on\ } \partial\Omega_D, \\
                \nabla t\cdot \mathbf{n} = \nu, &\text{on\ } \partial \Omega_N.
            \end{dcases}
        \end{equation*}
        \noindent Fixed $\mathbf{\hat{v}}\in V_s$ such that $\mathbf{\hat{v}} = -\nabla t$, then $\nabla\cdot \mathbf{\hat{v}} = 0$, $-\mathbf{\hat{v}}\cdot\mathbf{n}=\nu$ on $\partial\Omega_N$.\\
        Moreover, $\|\mathbf{\hat{v}}\|_{L^2(\Omega)} \leq \|t \|_{H^1(\Omega)} \leq C_2 \|\nu \|_{H^{-1/2}(\partial\Omega_N)}$ for some $C>0$, by continuity of the solution to the Poisson problem, and consequently $\|\mathbf{\hat{v}}\|_{V_s}^2 = \|\mathbf{\hat{v}}\|_{L^2(\Omega)}^2 + \|\nu\|_{H^{-1/2}(\partial\Omega_N}^2 \leq \left(1+C^2\right) \|\nu\|_{H^{-1/2}(\partial\Omega_N}^2$. \\
        Then, we can repeat the same steps as before and apply the Poincaré inequality to $\|\psi\|_{H^1(\Omega)}$ and Riesz representation theorem to $<\lambda,\nu>_{\partial\Omega_N}$:

        \begin{equation}
            \begin{split}
                \sup_
                {\scriptsize
                    \begin{split}
                    \left(\mathbf{v},\psi\right)&\in V_s\times W_\Lambda \\
                    \mathbf{v} &\neq \mathbf{0} \\
                    \psi &\neq 0
                \end{split}}
                \dfrac{b_i\left(\left(\mathbf{v},\psi\right),(\phi,\lambda)\right)}{\|(\mathbf{v},\psi)\|_{V_s\times W_\Lambda}}
                =
                \dfrac{b_i\left(\left(\mathbf{\tilde{v}},\tilde{\psi}_i\right),(\phi,\lambda)\right)}{\|\left(\mathbf{\tilde{v}},\tilde{\psi}_i\right)\|_{V_s\times W_\Lambda}} =
                \dfrac{ < \lambda,\nu >_{\partial\Omega_N} + \left\|\dfrac{d\tilde{\psi}_i}{ds}\right\|_{L^2(\Lambda)}^2 }{\left(C \|\nu\|_{H^{-1/2}(\partial\Omega_N)} ^2 + \|\tilde{\psi}_i\|_{W_\Lambda}^2\right)^{1/2}} \geq\\
                \geq \dfrac{ 1 }{\left(C \|\lambda\|_{H^{1/2}(\partial\Omega_N)} ^2+ \left(1+C_p^2\right)\left\|\dfrac{d\tilde{\psi}_i}{ds}\right\|_{L^2(\lambda)}^2\right)^{1/2}} \left( \|\lambda\|_{H^{1/2}(\partial\Omega_N)} ^2 + \left\|\dfrac{d\tilde{\psi}_i}{ds}\right\|_{L^2(\Lambda)}^2 \right)
                \geq C_2 \|\lambda\|_{H^{1/2}(\partial\Omega_N)}, \\ \forall (\phi,\lambda)\in W_s\times H^{1/2}(\partial\Omega_N), \ i=1,2.
            \end{split}
            \label{eq:lemma:b:inf-sup:2}
        \end{equation}
        
        \noindent This proves~\eqref{eq:lemma:b:inf-sup:mu}, which, together with~\eqref{eq:lemma:b:inf-sup:phi}, implies~\eqref{eq:lemma:b:inf-sup}.
	\end{proof}
	
	\begin{lemma}
		The operator $c_{ss}$ is positive semi-definite and symmetric.
	\end{lemma}
	
	\begin{proof}
        Symmetry is straightforward, since it is defined as the integral of a product.
        
		Consider now $\phi\in W_s$:
		\begin{equation*}
            \begin{split}
                c_{ss}\left(\phi,\phi\right) &= 
			    4\pi\epsilon_g\int_\Lambda \hat{\phi}^2 = 4\pi\epsilon_g \|\phi\|_{W_s}^2
				\geq 0,\quad \forall (\phi,\psi)\in W_\Lambda\times W_\Lambda.
			\end{split}
		\end{equation*}

        \noindent Thus, $c_{ss}$ is positive semi-definite.
        
	\end{proof}
	
	\begin{theorem}
		There exists a unique solution $\left( \left(\mathbf{D}_s, \Phi_\Lambda\right), \left(\Phi_s,\xi\right) \right)\in \left(V_s \times W_\Lambda\right) \times W_s\times H^{1/2}(\partial\Omega_N)$ to problem~\eqref{eq:3d-1d:weak:eq}.
	\end{theorem}
	
	\begin{proof}
		Since the right-hand side of the equation is continuous in $L^2$, $a$, $b_1$, $b_2$ and $c_{ss}$ are continuous, $a$ is coercive on $\ker{b}$ and positive semi-definite, $b_1$ and $b_2$ satisfy the $\inf-\sup$ condition~\eqref{eq:lemma:b:inf-sup} and $c$ is positive semi-definite and symmetric, the problem admits a unique solution (see \cite{brezzi1974existence}).
	\end{proof}
	
	\section{3D-1D problem on bifurcations}
	\label{section:bifurcations}
	
	In the previous sections we have derived the one-dimensional formulation of the electrostatic problem on a cylindrical gas domain, geometrically reduced to a straight line. In the case of electrical treeing, the defect has a much more complicated structure, possibly forked in some points. We can assume that each branch is very thin, as in Assumption~\ref{assumption:thin}, and extend Assumptions~\ref{assumption:constant_q},~\ref{assumption:fluctuations} and~\ref{assumption:null_flux} to each of them. As before, each branch is thus representable as a straight cylinder and reducible to a segment. The whole domain can thus be approximated as a one-dimensional graph where the problem on each segment is described by the system of equations~\eqref{eq:P2:primal:1d}. In this case we need to impose additional conditions at the graph nodes, accounting for 3D volumes at the intersection on branches and deriving reduced conditions.\\
	Let us start by integrating the governing equation~\eqref{eq:P2:primal:1} of the 3D problem in the gas, in primal form, over these subdomains, that we will denote by $\mathcal{J}$, and apply the divergence theorem:
	
	\begin{equation*}
		\int_\mathcal{J} \left(-\nabla\cdot\left(\epsilon_g \nabla\Phi_g\right)\right) = -\int_{\partial\mathcal{J}} \epsilon_g\nabla\Phi_g\cdot\mathbf{n}.
	\end{equation*}
	
	\begin{figure}
		\centering
		\resizebox{.35\textwidth}{!}
		{\begin{tikzpicture}
				\node[yshift=-2.65cm] {$\mathcal{J}$};
				\node (a) [cylinder,
				yshift=-4.25cm, 
				xshift=-2.5cm, 
				rotate=-150, 
				draw, 
				minimum height=5.6cm,
				minimum width=2cm, 
				line width = 1] {$\Omega_2$};
				\node (a) [cylinder,
				yshift=-4.25cm, 
				xshift=2.5cm, 
				rotate=-30,
				draw, 
				minimum height=5.6cm, 
				minimum width=2cm, 
				line width = 1] {$\Omega_3$};
				\node (a) [cylinder,
                fill = white,
				draw, 
				shape border rotate = 90,
				yshift = .1cm,
				minimum height=5.6cm,
				minimum width=2cm, 
				line width = 1] {$\Omega_1$};
				
				\draw [dashed] (0,2.5) -- (0,-2.1);
				\draw [->] (0,2.85) -- (0,1.3);
				\node [xshift=-.2cm, yshift=1.6cm] {$s$};
				\draw [->] (0,-2.15) -- (0,-1.5);
				\node [xshift=-.3cm, yshift=-1.9cm] {$\mathbf{n}_1$};
				
				\begin{scope}[transform canvas={xshift=-2.65cm,yshift=-4.25cm,rotate=-60}]
					\draw [dashed] (0,2.45) -- (0,-2.5);
					\draw [->] (0,2.45) -- (0,1.3);
					\node [xshift=-.2cm, yshift=1.6cm] {$s$};
					\node [xshift=.4cm, yshift=1.6cm] {$\mathbf{n}_2$};
				\end{scope}
				
				\begin{scope}[transform canvas={xshift=2.65cm,yshift=-4.25cm,rotate=60}]
					\draw [dashed] (0,2.45) -- (0,-2.5);
					\draw [->] (0,2.45) -- (0,1.3);
					\node [xshift=-.2cm, yshift=1.6cm] {$s$};
					\node [xshift=.3cm, yshift=1.6cm] {$\mathbf{n}_3$};
				\end{scope}
		\end{tikzpicture}}
		\caption{Junction among three cylinders $\Omega_1$, $\Omega_2$ and $\Omega_3$. The boundary of the junction volume $\mathcal{J}$ is the union of the bases of the three cylinders and the lateral surface of the 3D domain.}
		\label{figure:intersection}
	\end{figure}
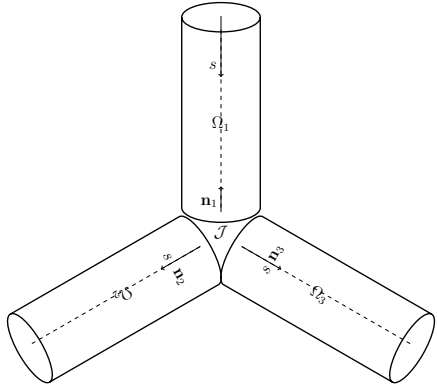
	
	\noindent The boundary of the junction $\mathcal{J}$ is given by the union of the bases of three cylinders $\Omega_i, \ i=1,2,3,$ representing the branches connected by it, and a small portion of lateral surface that we denote by $ \mathcal{S} $: $\partial\mathcal{J} = \left(\mathcal{J}\cap \Omega_1\right) \cup \left(\mathcal{J}\cap \Omega_2\right) \cup \left(\mathcal{J}\cap \Omega_3\right) \cup \mathcal{S}$, as in Figure~\ref{figure:intersection}. We call $\mathbf{n}_i$ the unit vector with direction parallel to the axis of each cylinder $\Omega_i$, pointing outwards from $\mathcal{J}$. If $\mathbf{n}_i$ has the same direction as $\mathbf{s}$ along the centerline $\Lambda_i$ of $\Omega_i$, we say that $\Lambda_i$ becomes an \textit{incoming edge} for the graph, with respect to the considered intersection node, otherwise it is an \textit{outgoing edge}.  In the reduction to 1D, the volume $\mathcal{J}$ is collapsed into a point connected to more than two edges, called \textit{bifurcation node}. For each bifurcation node $v_j$, $j=1,\dots,\mathcal{N}_\text{bif}$, where $\mathcal{N}_\text{bif}$ is the number of bifurcations, we denote by $\mathcal{E}^+_j$ and $\mathcal{E}^-_j$ the sets of its incoming and outgoing edges, respectively. In the example represented by Figure~\ref{figure:intersection}, the vectors $\mathbf{n}_i$, $i=1,2,3$, are orthogonal to the bases of cylinders forming the surface of $\mathcal{J}$, with the same direction and $\mathbf{s}$ in $\Omega_2$ and $\Omega_3$ and opposite direction to $\mathbf{s}$ in $\Omega_1$. We obtain:
	
	\begin{equation*}
		-\int_{\mathcal{J}\cap \Omega_i} \epsilon_g \nabla\Phi_g\cdot \mathbf{n} = \int_{\mathcal{J}\cap \Omega_i} \epsilon_g \nabla\Phi_g\cdot \mathbf{n}_i =
		\begin{dcases}
			\int_{\mathcal{J}\cap \Omega_i} \epsilon_g \dfrac{\partial\Phi_g}{\partial s}, & i=1,\\
			-\int_{\mathcal{J}\cap \Omega_i} \epsilon_g \dfrac{\partial \Phi_g}{\partial s}, & i=2,3.
		\end{dcases}
	\end{equation*}
	
	\noindent If we substitute the potential $\Phi_g$ with the splitting introduced in equation~\eqref{eq:splitting_phi_g}, we obtain:
	
	\begin{equation*}
		\int_{\mathcal{J}\cap \Omega_i} \epsilon_g \dfrac{\partial \Phi_g}{\partial s} = \int_{\mathcal{J}\cap \Omega_i} \epsilon_g \dfrac{d \Phi_\Lambda}{d s} + \int_{\mathcal{J}\cap \Omega_i} \epsilon_g \dfrac{d \Phi_r}{d s}\phi.
	\end{equation*}
	
	\noindent Assuming that the branches can be approximated by cylinders, then the bases are circles of radius $R_i$ and the integral becomes:
	
	\begin{equation*}
		\int_{\mathcal{J}\cap \Omega_i} \epsilon_g \dfrac{d \Phi_\Lambda}{d s} + \int_{\mathcal{J}\cap \Omega_i} \epsilon_g \dfrac{d \Phi_r}{d s}\phi = \epsilon_g \dfrac{d\Phi_\Lambda}{ds}\pi R_i^2 + \epsilon_g\dfrac{d\Phi_r}{ds}\int_0^{2\pi}\int_0^R r^2 r dr = \pi R_i^2\epsilon_g \dfrac{d\Phi_\Lambda}{ds} + \dfrac{\pi}{2}R_i^4 \epsilon_g\dfrac{d\Phi_r}{ds}.
	\end{equation*}
	
\noindent Let us denote by $\Phi_{\Lambda,i}$ and $\Phi_{r,i}$ the corresponding unknowns on each branch $\Omega_i$, $a_i=\pi R_i^2\epsilon_g$ and $b_i = \dfrac{\pi}{2}R_i^4 \epsilon_g\dfrac{d\Phi_{r,i}}{ds}, \ i=1,2,3$. Then, equation~\eqref{eq:P2:primal:1}, integrated over $\mathcal{J}$, becomes:
	
	\begin{equation*}
		a_1\dfrac{d\Phi_{\Lambda,1}}{ds} + b_1 -  a_2\dfrac{d\Phi_{\Lambda,2}}{ds} - b_2 - a_3\dfrac{d\Phi_{\Lambda,3}}{ds} - b_3 =
		\int_\mathcal{J} f - \int_\mathcal{S} \epsilon_g\nabla\Phi_g\cdot \mathbf{n}.
	\end{equation*}
	
	\noindent The quantities $b_i$, dependent on $R_i^4$, $\int_\mathcal{J} f$, proportional to $R_i^3$, are negligible, compared to $a_i$. Moreover, according to the splitting~\eqref{eq:splitting_phi_g} of $\Phi_g$, the normal component $\nabla\Phi_g\cdot\mathbf{n}$ is only dependent on $R_i^2$. If we integrate it on $\mathcal{S}$, whose area is of order $R_i^2$, the term $- \int_\mathcal{S} \epsilon_g\nabla\Phi_g\cdot \mathbf{n}$ is of order $R_i^4$ and, thus, negligible. Then, the condition on the junctions in the case of an arbitrary number of edges becomes:
	
	\begin{equation*}
		\sum_{i\in \mathcal{E}^+_j} a_i \dfrac{d\Phi_{\Lambda,i}}{ds} = \sum_{i\in \mathcal{E}^-_j} a_i \dfrac{d\Phi_{\Lambda,i}}{ds}, \quad j=1,\dots,\mathcal{N}_\text{bif}.
	\end{equation*}
	
	\noindent The final problem reads as follows:
	
	\begin{equation}
		\begin{dcases}
			\epsilon_s^{-1} \mathbf{D}_s  + \nabla\Phi_s = 0, & \text{in}\ \Omega, \\
			\nabla\cdot \mathbf{D}_s - 4\pi \epsilon_g\sum_{i\in\mathcal{E}}(\Phi_{\Lambda,i} - \hat{\Phi}_s)\delta_{\Lambda,i} = -2\pi\sum_{i\in\mathcal{E}} R_ig_i\delta_{\Lambda,i}, & \text{in}\ \Omega, \\
			-\pi R_i^2 \dfrac{d}{ds}\left(\epsilon_g\dfrac{d\Phi_{\Lambda,i}}{d s}\right) + 4\pi \epsilon_g \left(\Phi_{\Lambda,i} - \hat{\Phi}_s\right) 
			= \pi R_i^2 \dfrac{q_i}{\epsilon_0}, & \text{on}\ \Lambda_i, \ i=1,\dots,\mathcal{N}_e,\\
			\Phi_{\Lambda,i} = \hat{\bar{\Phi}}_g, & \text{on\ }\mathcal{N}_D, \\
			\dfrac{d\Phi_{\Lambda,i}}{ds} = 
			-\frac{2}{\epsilon_g} g, & \text{on\ }\mathcal{N}_N, \\
			\sum_{i\in \mathcal{E}^+_j} a_i \dfrac{d\Phi_{\Lambda,i}}{ds} = \sum_{i\in \mathcal{E}^-_j} a_i \dfrac{d\Phi_{\Lambda,i}}{ds}, & \text{on\ }\mathcal{N_\text{bif}}, \ j = 1,\dots,\mathcal{N}_\text{bif},\\
			\Phi_s = \bar{\Phi}_s, & \text{on}\ \partial\Omega_D, \\
			\mathbf{D}_s\cdot\mathbf{n} = \nu, & \text{on}\ \partial\Omega_N,
		\end{dcases}
		\label{eq:3d-1d:graph}
	\end{equation}
	
	\noindent where we have denoted by $\mathcal{N}_e$ the number of edges of the one-dimensional graph, $\mathcal{N}_N$ the set of nodes of $\Lambda$ where Neumann conditions are imposed, $\mathcal{N}_D$ the set of Dirichlet nodes of $\Lambda$ and $\mathcal{N}_\text{bif}$ the number of bifurcation nodes.
	
	\section{Numerical methods}
	\label{section:NumericalMethods}
	
	We present a discretization based on Finite Element Methods (FEM) for the numerical solution of problem~\eqref{eq:3d-1d:graph}. We employ independent meshes for the one and three-dimensional domains, consisting of segments and tetrahedra, respectively. In particular, in presence of bifurcations, the one-dimensional domain can be seen as a quantum graph, i.e. a graph endowed with an implicit metric structure an a differential operator defined on its edges and vertices (see~\cite{berkolaiko2013introduction} for details). \textcite{arioli2018finite} proposed an extension of FEM on such domains, where a partition of each edge is obtained by the addition of further nodes, thus creating a structure called \textit{extended graph}. The three-dimensional domain $\Omega$, instead, is partitioned in a tetrahedral conforming mesh on which we apply mixed FEM \cite{boffi2013mixed}.
	
	The weak formulation for the complete problem~\eqref{eq:3d-1d:graph} reads as follows:\\
	
	Find $(\mathbf{D}_s,\ \Phi_s,\ \Phi_\Lambda,\ \lambda_N,\ \lambda_D) \in V_s\times W_s \times W_\Lambda \times H^{-1/2}(\partial\Omega_N) \times \mathbb{R} $ such that:
	
	\begin{equation*}
		\begin{dcases}
			\mathcal{A}(\mathbf{D}_s,\mathbf{v}) + \mathcal{B}(\mathbf{v},\Phi_s) = \mathcal{F}_D(\mathbf{v}), & \forall \mathbf{v}\in V_s, \\
			\mathcal{B}(\mathbf{D}_s,\phi) + c(\Phi_\Lambda-\hat{\Phi}_s,\hat{\phi}) + l(\lambda_N,\mathbf{v}) = \mathcal{G}(\hat{\phi}), & \forall \phi\in W_s, \\
			\mathcal{A}_\Lambda(\Phi_\Lambda,\psi) + c(\Phi_\Lambda - \hat{\Phi}_s,\psi) + d(\lambda_D,\psi) = \mathcal{F}(\psi), & \forall \psi\in W_\Lambda, \\
			l(\xi_N,\mathbf{D}_s) = \mathcal{F}_N(\xi_N), & \forall \xi_N\in H^{-1/2}(\partial \Omega_N), \\
			d(\xi_D,\Phi_\Lambda) = \mathcal{F}_{D,\Lambda}(\xi_D), & \forall \xi_D\in \mathbb{R}
		\end{dcases}
	\end{equation*}
	
	\noindent where we have defined 
	
	\begin{equation*}
		\begin{alignedat}{5}
			&l :& \ V_s\times H^{-1/2}(\partial \Omega_N) &\longrightarrow \mathbb{R}, \qquad& l(\nu,\mathbf{v}) &= \int_{\partial\Omega_N} \nu \mathbf{v}\cdot\mathbf{n}; \\
			&d\ :& \ W_s\times \mathbb{R} &\longrightarrow \mathbb{R}, \qquad& d(\nu,\psi) &= \nu(0)\psi; \\
			&\mathcal{F}_D\ :& \ V_s &\longrightarrow \mathbb{R}, \qquad& \mathcal{F}_D(\mathbf{v}) &= -\int_{\partial\Omega_D} \bar{\Phi}_s \mathbf{v}\cdot\mathbf{n}; \\
			&\mathcal{G}\ :& \ W_s &\longrightarrow \mathbb{R}, \qquad& \mathcal{G}(\phi) &= <G,\phi>; \\
			&\mathcal{F}\ :&\ W_\Lambda &\longrightarrow \mathbb{R}, \qquad& \mathcal{F}(\psi) &= <F,\psi> \\
			&\mathcal{F}_N \ :&\ H^{-1/2}(\partial \Omega_N) &\longrightarrow \mathbb{R}, \qquad& \mathcal{F}_N(\nu) &= \int_{\partial\Omega_N} \nu \bar{\mu};\\
			&\mathcal{F}_{D,\Lambda}\ :& \ \mathbb{R} &\longrightarrow \mathbb{R}, \qquad & \mathcal{F}_{D,\Lambda}(\nu) &= \hat{\bar{\Phi}}_g\nu,		
		\end{alignedat}
	\end{equation*}
	
	\noindent while $\mathcal{A}$, $\mathcal{B}$, $c$, $\mathcal{A}_\Lambda$, $F$, $G$ and all the aforementioned spaces were introduced in Section~\ref{section:reduced}.\\
	Here we have imposed the Neumann boundary conditions on the 3D domain and Dirichlet ones on the 1D by means of Lagrange multipliers $\lambda_N$ and $\lambda_D$. Notice that the conditions at the bifurcations are natural for the primal formulation of the 1D problem.
	
	\subsection{Discrete weak formulation}
	\label{section:NM:DWF}
	For the discretization of the one-dimensional problem on $\Lambda$ (see~\cite{benzi_golub_liesen_2005}) we start by introducing a set of $n_e+1$ equispaced points $x_j^e$, $j=0,1,\dots,n_e$, on each edge $e$, partitioning it into $n_e$ intervals of length $k_e$, and associate to each internal point $x_j^e,\ j=1,\dots,n_e-1$ the standard hat basis function $\psi_j^e$. We also introduce the functions $\phi_v,\ v\in\mathcal{N}$, such that $\phi_v = 1$ on the vertex $v$, $\phi_v = 0$ on all the other nodes on the neighboring edges of $v$ and is piecewise linear on each segment of the mesh. This set of functions forms a basis for the discrete space of piece-wise polynomial functions of order 1 on the neighboring segments to each bifurcation node.
	
	Then, we define the FEM space on $\Lambda$ as the direct sum of spaces spanned by all the sets of basis functions $\psi_j^e,\ e\in\mathcal{E}$ and $\phi_v,\ v\in\mathcal{N}$:
	
	\begin{equation*}
		W_{\Lambda,k} = \left( \bigoplus_{e\in\mathcal{E}} W_{\Lambda,k}^e \right) \oplus \text{span}\{\phi_v\}_{v\in\mathcal{N}},
	\end{equation*}
	
	\noindent where $W_{\Lambda,k}^e = \left\{w\in W_\Lambda \ : \ w\big|_{\left[x_j^e,x_{j+1}^e\right]} \in \mathbb{P}^1, \ j=0,1,\dots,n_e-1\right\}$ denotes the finite-dimensional space of piecewise linear functions on the edge $e$.
	
	Finally, let us denote by $\{\psi_i\}_{i=1}^{\text{dim}(W_{\Lambda,k})}$ the set of all the basis functions of $W_{\Lambda,k}$, whose linear combinations characterize all the elements in this finite-dimensional space:
	
	\begin{equation*}
		v = \sum_{i=1}^{\text{dim}(W_{\Lambda,k})} v_i\psi_i, \quad \forall v\in W_{\Lambda,k}.
	\end{equation*}
	
	\noindent In particular, it holds for our discrete approximation of the unknown $\Phi_\Lambda$, whose expansion coefficients we denote by $\mathbf{\Phi_{\Lambda,k}} \in\mathbb{R}^{n_\Lambda}$, with $n_{\Lambda} = \text{dim}(W_{\Lambda,k})$.
	
	For the three-dimensional problem we employ the standard Raviart-Thomas dual mixed FEM spaces of lowest order on a tetrahedral mesh $\mathcal{T}_h = \{T\}$ made of $n_T$ elements, discussed in \cite{babuvska2003mixed}. The considered finite-dimensional subspaces of $V_s$, $W_s$ and $H^{-1/2}(\partial\Omega_N)$ are, respectively:
	
	\begin{equation*}
		\begin{aligned}
			V_{s,h} & = \left\{ \mathbf{v}\in V_s \ : \ \mathbf{v}\big|_{T} \in \mathbb{RT}^0(T) \ \forall T\in\mathcal{T}_h \right\}, \\
			W_{s,h} & = \left\{ \phi \in W_s \ : \ \phi\big|_{T}\in \mathbb{P}^0(T) \ \forall T\in\mathcal{T}_h \right\}, \\
			H^{-1/2}_h & = \left\{ \nu \in W_s \ : \ \nu\big|_{\Gamma_j}\in \mathbb{P}^0(\Gamma_j), \ j=1,\dots,n_\Gamma \right\},
		\end{aligned}
	\end{equation*}
	
	\noindent where $\{\Gamma_j\}_{j=1}^{n_\Gamma}$ is the partition of $\partial \Omega_N$ induced by the 3D mesh $\mathcal{T}_h$, $\mathbb{RT}^0$ is the Raviart-Thomas (RT) of lowest degree \cite{raviart1977mixed} and $\mathbb{P}^0$ the space of polynomials of order 0. Then, $\dim(V_{s,h})=n_F$, $\dim(W_{s,h}) = n_T$ and $\dim(H_h^{-1/2})=n_\Gamma$, where $n_F$, $n_T$ and $n_\Gamma$ are the number of faces, elements of 3D mesh $\mathcal{T}_h$ and elements of the partition $\{\Gamma_j\}_{j=1}^{n_\Gamma}$ of the boundary, respectively.\\
	Let $\{\mathbf{v}_i\}_{i=1}^{n_F}$, $\{\phi_i\}_{i=1}^{n_T}$ and $\{\nu_i\}_{i=1}^{n_\Gamma}$ be the bases of the spaces $V_{s,h}$, $W_{s,h}$ and $H^{-1/2}_h$, respectively, and denote by $\mathbf{D}_{s,h}$, $\bm{\Phi}_{s,h}$ and $\bm{\lambda}_{N,h}$ the vectors of expansion coefficients of the discrete approximations of $\mathbf{D}_s$, $\Phi_s$ and $\lambda_N$ with respect to these bases. Then, the discrete weak formulation with Lagrange multipliers of the mixed-dimensional problem~\eqref{eq:3d-1d:graph} is the following:\\
	
	Find $(\mathbf{D}_{s,h},\ \bm{\Phi}_{s,h},\ \Phi_{\Lambda,k},\ \bm{\lambda}_{N,h},\ \lambda_D) \in \mathbb{R}^{n_T} \times \mathbb{R}^{n_T} \times \mathbb{R}^{n_\Lambda} \times \mathbb{R}^{n_\Gamma} \times \mathbb{R} $ such that:
	
	\begin{equation*}
		\begin{dcases}
			\mathcal{A}(\mathbf{D}_{s,h},\mathbf{v}_h) + \mathcal{B}(\mathbf{v}_h,\bm{\Phi}_{s,h}) = \mathcal{F}_D(\mathbf{v}_h), & \forall \mathbf{v}_h\in \mathbb{R}^{n_T}, \\
			\mathcal{B}(\mathbf{D}_{s,h},\bm{\phi}_h) + c_{\Lambda s}(\bm{\Phi}_{\Lambda,k},\hat{\bm{\phi}}_h) -c_{ss}(\hat{\bm{\Phi}}_{s,h},\hat{\bm{\phi}}_h) + l(\bm{\lambda}_{N,h},\mathbf{v}_h) = \mathcal{G}(\hat{\bm{\phi}}_h), & \forall \bm{\phi}_h\in \mathbb{R}^{n_T}, \\
			\mathcal{A}_\Lambda(\bm{\Phi}_{\Lambda,k},\bm{\psi}_k) + c_{\Lambda\Lambda}(\bm{\Phi}_{\Lambda,k},\psi_k) - c_{\Lambda s}(\hat{\bm{\Phi}}_{s,h},\bm{\psi}_k) + d(\lambda_D,\bm{\psi}_k) = \mathcal{F}(\bm{\psi}_k), & \forall \bm{\psi}_k\in \mathbb{R}^{n_\Lambda} , \\
			l(\bm{\nu}_{N,h},\mathbf{D}_{s,h}) = \mathcal{F}_N(\bm{\nu}_{N,h}), & \forall \bm{\nu}_{N,h}\in \mathbb{R}^{n_\Gamma} , \\
			d(\nu_D,\bm{\Phi}_{\Lambda,k}) = \mathcal{F}_{D,\Lambda}(\nu_D), & \forall \nu_D\in \mathbb{R}
		\end{dcases}
	\end{equation*}
	
	\noindent and the corresponding linear system is given by:\\
	
	\begin{equation}
		\begin{bmatrix}
			A_s & B^T              & 0                  & L^T & 0 \\
			B   & C_{ss}           & C_{\Lambda s}^T    & 0   & 0 \\
			0   & C_{\Lambda s}    & A_\Lambda          & 0   & L_D^T \\
			L   & 0                & 0                  & 0   & 0 \\
			0   & 0                & L_D                & 0   & 0
		\end{bmatrix}
		\begin{bmatrix}
			\mathbf{D}_{s,h} \\
			\bm{\Phi}_{s,h}	\\
			\bm{\Phi}_{\Lambda,k} \\
			\bm{\lambda}_{N,h} \\
			\lambda_D
		\end{bmatrix}
		=
		\begin{bmatrix}
			\mathbf{F}_D \\
			\mathbf{G} \\
			\mathbf{F} \\
			\mathbf{F}_N \\
			F_{D,\Lambda}
		\end{bmatrix}
		\label{eq:3d-1d:matrix}
	\end{equation}
	
	\noindent where the the blocks of the system matrix are defined as follows:
	
	\begin{equation*}
        \begin{split}
    		&\begin{alignedat}{7}
    			&[A_s]_{i,j} &= \mathcal{A}(\mathbf{v}_i,\mathbf{v}_j), \quad & i,j=1,\dots,n_T, \qquad
    			&[B]_{i,j} & = - \int_\Omega \phi_j \nabla\cdot\mathbf{u}_i, \quad & i,j=1,\dots,n_T, \\
    			&[C_{\Lambda s}]_{i,j} &= c_{\Lambda s}(\psi_i,\phi_j), \quad & i,j=1,\dots,n_\Lambda, \qquad
    			&[C_{ss}]_{i,j} & = -c_{ss}(\hat{\phi}_i,\hat{\phi}_j), \quad & i,j=1,\dots,n_T,\\
    			&[L]_{i,j} &= l(\nu_i,\mathbf{v}_j), \quad & 
    			\begin{split}
    				i&=1,\dots,n_\Gamma, \\
    				j&=1,\dots,n_T,
    			\end{split} \quad
    			&[\mathbf{L}_D]_{i} & = -d(\lambda_D,\psi_i), \quad & i=1,\dots,n_\Lambda,
    		\end{alignedat}\\
    		&[A_\Lambda]_{i,j} = -\mathcal{A}_\Lambda(\psi_i,\psi_j) - c_{\Lambda\Lambda}(\psi_i,\psi_j), \quad i,j=1,\dots,n_\Lambda.
        \end{split}
	\end{equation*}

    \noindent This system is symmetric and the blocks on the first row arise from the standard mixed FEM discretization of the 3D problem~\eqref{eq:P1:dual:1}-~\eqref{eq:P1:dual:6} in the dielectric domain, while the coupling terms between the two domains, given by the operators $c_{\Lambda\Lambda},\ c_{ss}$ and $c_{\Lambda s}$ appear inside the block
    $\begin{bmatrix}
        C_{ss}        &C_{\Lambda s}^T \\
        C_{\Lambda s} &A_\Lambda
    \end{bmatrix}$.
	
	\subsection{Numerical solver}
	\label{section:NM:solver}
	We solve the discrete linear system~\eqref{eq:3d-1d:matrix} iteratively with GMRES~\cite{saad1986gmres}, and tackle the issue of possible bad conditioning of the matrix and presence of zero blocks on the diagonal using a suitable a preconditioner $P$.
	
	We exploit a preconditioner discussed in \textcite{benzi_golub_liesen_2005} for saddle-point problems of type
	
	\begin{equation*}
		\begin{bmatrix}
			M_1 & M_2^T \\
			M_2 & M_3
		\end{bmatrix},
	\end{equation*}
	
	\noindent where $M_1$ is positive definite, $M_2$ has maximum column
 rank and $M_3$ is negative semi-definite:
	
	\begin{equation*}
		P =
		\begin{bmatrix}
			M_1 & 0 \\
			0 & -(M_3-M_2 M_1^{-1} M_2^T)
		\end{bmatrix}^{-1}.
	\end{equation*}
	
	\noindent If we apply this to our problem, where 
    \begin{equation*}
         M_1 = A_s,\ M_2 = [B,\ 0,\ L,\ 0]^T\ \text{and}\ M_3 = 
        \begin{bmatrix}
			C_{ss}           & C_{\Lambda s}^T    & 0   & 0 \\
			C_{\Lambda s}    & A_\Lambda          & 0   & L_D^T \\
			0                & 0                  & 0   & 0 \\
			0                & L_D                & 0   & 0
		\end{bmatrix},        
    \end{equation*}
    \noindent we obtain:
	
	\begin{equation*}
		P =
		\begin{bmatrix}
			A_s & 0      & 0 \\
			0   & -\Sigma & -E^T \\
			0   & -E      & 0
		\end{bmatrix}^{-1},
	\end{equation*}
	
	\noindent where the blocks $\Sigma $ and $E$ are defined as:
	
	\begin{equation*}
		\Sigma =
		\begin{bmatrix}
			C_{ss}-BA_s^{-1}B^{T} & C_{\Lambda s}^T & -B A_s^{-1} L^T \\
			C_{\Lambda s}         & A_\Lambda       & 0               \\
            -LA_s^{-1}B^T         & 0               & -L A_s^{-1}L^T    
		\end{bmatrix},
		\qquad
		E = 
		\begin{bmatrix}
			0 & L_D & 0
		\end{bmatrix}.
	\end{equation*}
	
	\noindent Since the inverse of a diagonal block matrix is a diagonal block matrix with diagonal blocks given by the inverse of the original diagonal blocks, we need to invert the bottom-right block of the resulting matrix, which is again possibly ill-conditioned and saddle-point. Therefore, we adopt the same strategy, thus obtaining the following block-diagonal preconditioner in the end:
	
	\begin{equation*}
		P = 
		\begin{bmatrix}
			A_s^{-1} & 0            & 0 \\
			0        & -\Sigma^{-1} & 0 \\
			0        & 0            & -E\Sigma^{-1}E^T
		\end{bmatrix}.
	\end{equation*} 
	
	\noindent A deep investigation of the properties of this final preconditioner is still an ongoing work.  However, as we will see in the next section, the number of iterations needed by the GMRES solver is moderate and allows the solution of the presented coupled problem in reasonable time.
	
	\section{Results}
	\label{section:Results}
	We conclude this work by presenting some tests on a simple three-dimensional cylindrical domain, coupled with different one-dimensional geometries approximating thin inclusions. In particular, in Section~\ref{section:TC1} we analyze a straight line going from the center of the top basis of the cylinder to the center of the bottom basis, a simple geometry that allows us to compare the numerical solution to the reduced 1D-3D problem~\eqref{eq:3d-1d:1}-~\eqref{eq:3d-1d:7} to that of the original 3D-3D problem~\eqref{eq:3d-3d}. Knowing the exact solution, we can validate the model as the radius of $\Omega_g$ tends to 0 and evaluate the dependence of the accuracy on the size of the elements in the one-dimensional mesh. Then, in Section~\ref{section:TC2}  we observe the effect of an immersed tip inside the 3D domain and in Section~\ref{section:TC3}  we move to the solution of a very complex electrical treeing.
	In all the tests the domain $\Omega$ consists of a cylinder whose bases are parallel to the $x,y$ plane and centered in $(0.5,0.5,0)$ and $(0.5,0.5,1)$, respectively, with radius $R_s=1$. For all these tests we have considered the same dielectric constant in both domains: $\epsilon_s=\epsilon_g=1$.
	
	\subsection{TC1: Cylinder and straight line}
	\label{section:TC1}

	We start with a simple geometry, represented in Figure~\ref{figure:TC1:mesh}, consisting of the cylinder introduced above, discretized with a finer mesh close to the centerline, and a line with endpoints in the centers of the bases. In particular, we consider elements with maximum radius $10^{-2}$ in the proximity of $\Lambda$ and $10^{-1}$ in the rest of the domain. The one-dimensional domain is, instead, partitioned into 100 segments. We suppose that this geometry is the mixed-dimensional approximation of two three-dimensional domains where the innermost one, representing the gas, is a cylinder of given radius $R$. Thus, we consider the following exact solution, adapted from \cite{gjerde2021mixed} on $\Omega_s$ and $\Omega_g$:
	
	\begin{equation*}
		\begin{alignedat}{2}
			\Phi_s^\text{ex}(r,s) &= \left[1-\log\left(\dfrac{r}{R}\right)\right]R, \quad& \text{in\ }\Omega_s, \\
			\Phi_g^\text{ex}(r,s) &= \dfrac{1}{2}\dfrac{r^2}{R} + \dfrac{1}{2}R, \quad& \text{in\ }\Omega_g, \\
			\mathbf{D}_s^\text{ex}(r,s) &= \dfrac{R}{r}\mathbf{r}, \quad & \text{in\ }\Omega_s,\\
			\mathbf{D}_g^\text{ex}(r,s) &= \dfrac{r}{R}\mathbf{r}, \quad & \text{in\ }\Omega_g, \\
		\end{alignedat}
	\end{equation*}
	
	\noindent Notice that all the quantities only depend on the radial coordinate $r$ and the electric fields only have non-zero radial component, the continuity condition of the potential on the interface is satisfied and the jump of the normal component of the electric field is zero. Moreover, if we decompose the gas potential $\Phi_g^\text{ex}$ as in equation~\eqref{eq:splitting_phi_g}, we obtain $\Phi_r = \Phi_\Lambda = \dfrac{1}{2R}$ on $\Lambda$.\\
	From this exact solution we can compute the boundary conditions and forcing terms of our specific problem, obtaining:
	
	\begin{equation}
		\begin{cases}
			\mathbf{D}_s  + \nabla\Phi_s= 0, & \text{in}\ \Omega, \\
			\nabla\cdot\mathbf{D}_s - 4\pi (\Phi_\Lambda - \hat{\Phi}_s)\delta_\Lambda = 0, & \text{in}\ \Omega, \\
			-\pi R^2 \dfrac{d}{ds}\left(\dfrac{d\Phi_\Lambda}{d s}\right) + 4\pi \left(\Phi_\Lambda - \hat{\Phi}_s\right) 
			= 2\pi R, & \text{on}\ \Lambda,\\
			\Phi_\Lambda(0) = \Phi_\Lambda(1) = \dfrac{1}{2}R, \\
			\Phi_s = \Phi_s^{\text{ex}}, & \text{on}\ \partial\Omega_{D}, \\
			\mathbf{D}_s\cdot\mathbf{n} = R, & \text{on}\ \partial\Omega_N.
		\end{cases}
		\label{eq:TC1}
	\end{equation}
	
	\noindent Here the Neumann boundary $\partial\Omega_N$ coincides with the two bases of the cylinder and the Dirichlet boundary $\partial\Omega_D$ with its lateral surface.
	
	\begin{figure}
		\centering\includegraphics[width=.25\textwidth]{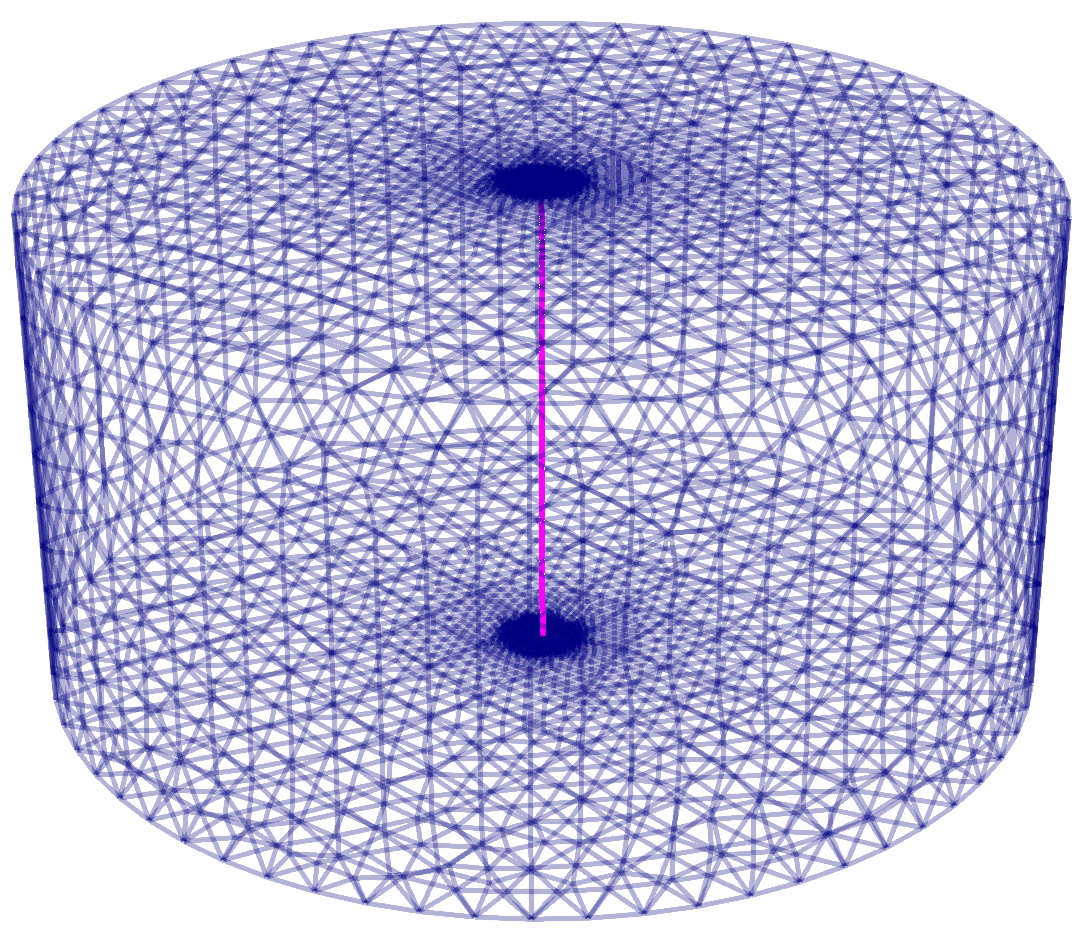}
		\caption{\textbf{TC1 -} Discretizations of the three-dimensional (blue) and one-dimensional (purple) domains.}
		\label{figure:TC1:mesh}
	\end{figure}
	\noindent In Figures~\ref{figure:TC1:top}-\ref{figure:TC1:section} we can observe the computed potential and electric field in the three-dimensional dielectric domain. The radial decrease of the potential on the top basis, expected as a consequence of the charge distribution along the axis of the cylinder, is shown in Figure~\ref{figure:TC1:top}, is also observed in Figure~\ref{figure:TC1:front}, with the expected value on the boundary. The major difficulty in the approximation is represented by the electric field, which presents a singularity on the centerline $\Lambda$, with magnitude going to infinity. Indeed, we can observe in Figure~\ref{figure:TC1:top} that the arrows representing its direction and magnitude, become much larger closer to the center, and in Figure~\ref{figure:TC1:lines} its magnitude shows a rapid increase there. \\
	These results are comparable to the numerical solution of the equidimensional original three-dimensional problem resolved by a fine grid and present qualitatively similar radial trend for the potential and direction and intensity of the electric field. For this comparison we have fixed the radius of the inner cylinder as $R=10^{-2}$. We can observe in Figure~\ref{figure:TC1:3d1d} that the region where the potential has the highest values in the 3D problem is wider because it coincides with the original gas-filled cylinder, while for the mixed-dimensional solution it is concentrated along the line $\Lambda$.
	
	\begin{figure}
		\centering\includegraphics[width=.5\textwidth]{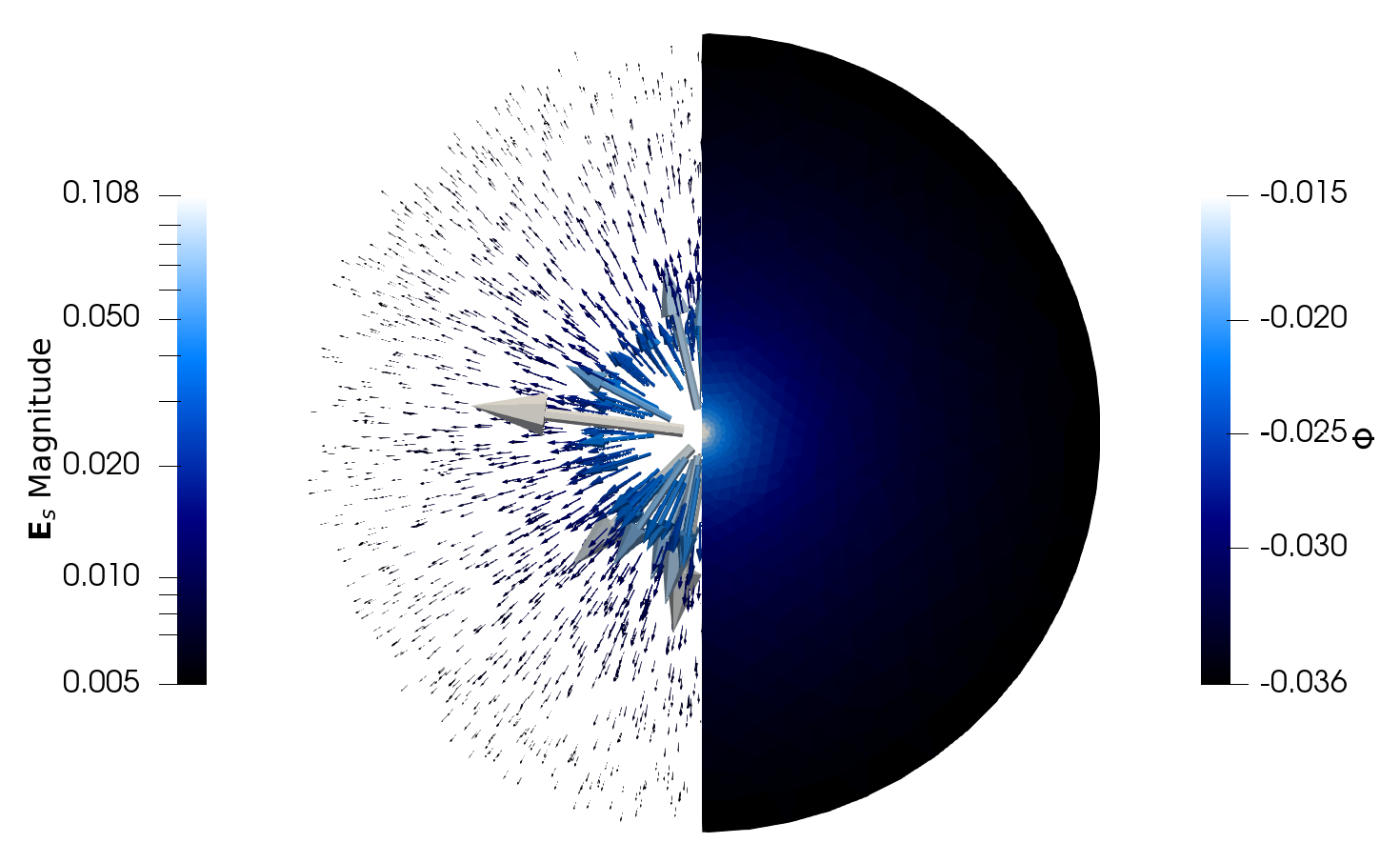}
		\caption{\textbf{TC1 -} Computed potential $\Phi$ (right half) and electric field $\mathbf{D}_s$ (left half) in the dielectric domain, osbserved from the top basis of the domain $\Omega$.}
		\label{figure:TC1:top}
	\end{figure}

	\begin{figure}
		\centering
		\subfloat[Electric potential\label{figure:TC1:front}] {\includegraphics[width=.4\textwidth]{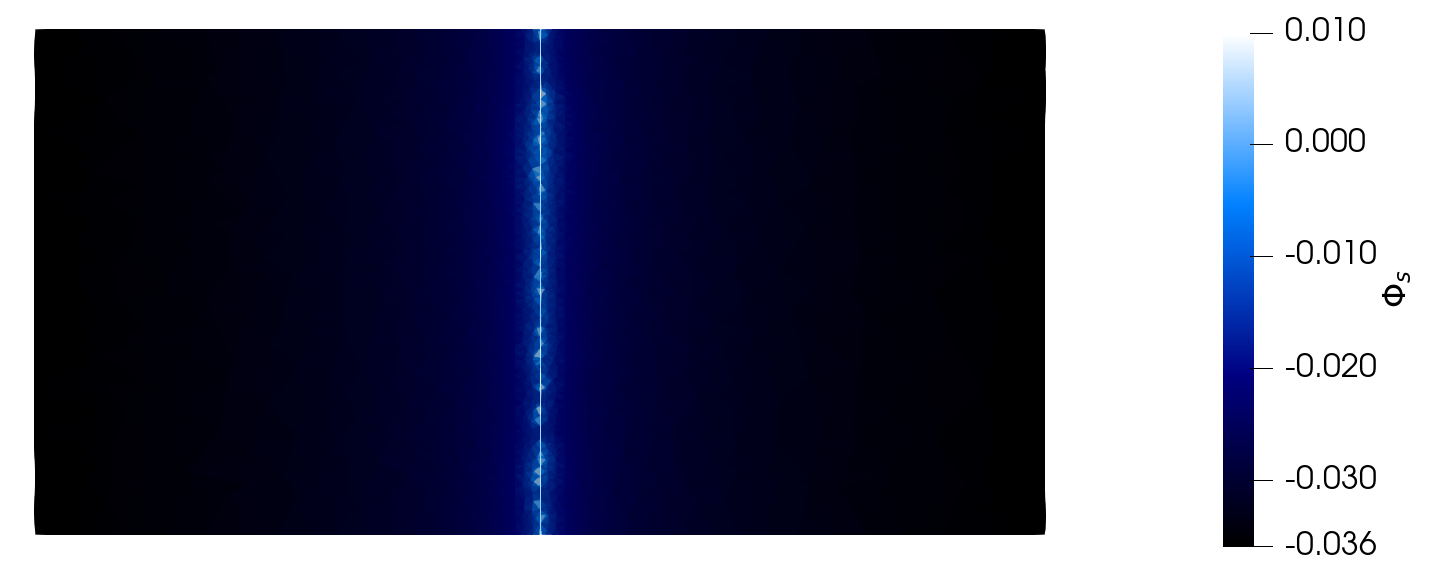}\quad}
		\subfloat[Electric field\label{figure:TC1:lines}] {\includegraphics[width=.4\textwidth]{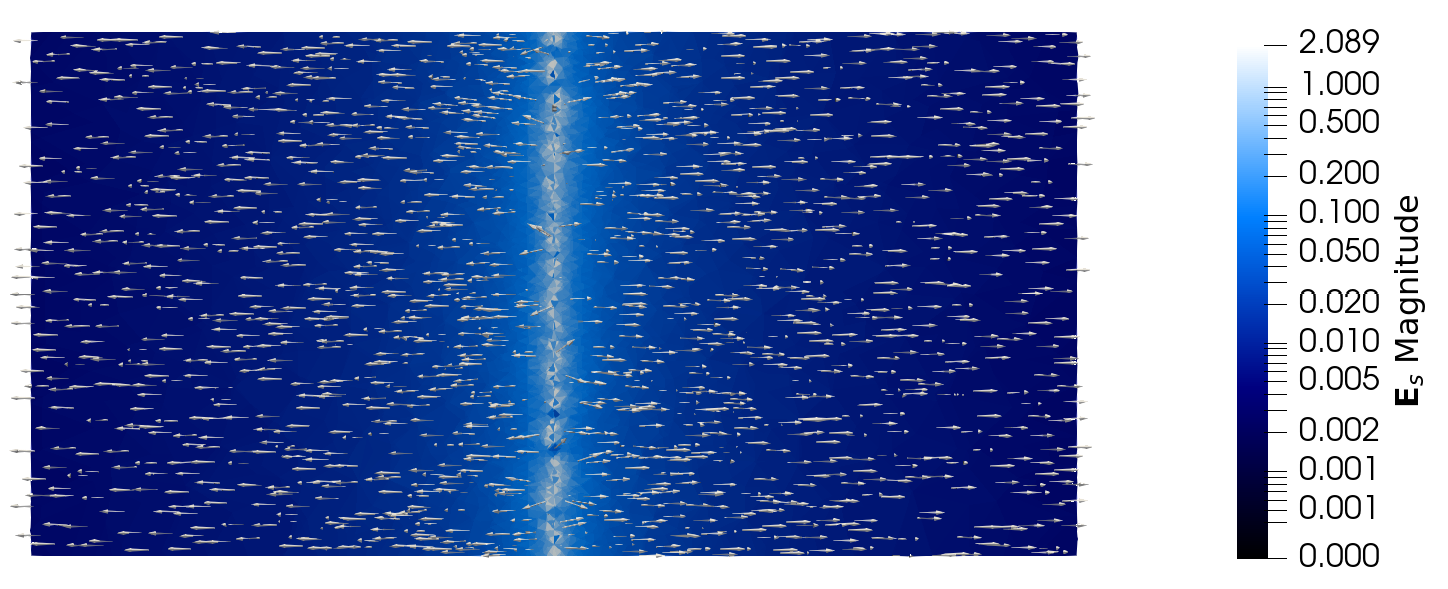}}
		\caption{\textbf{TC1 -} Computed potential $\Phi_s$ and electric field $\mathbf{D}_s$ in the 3D dielectric domain on a longitudinal section of $\Omega$. The magnitude of $\mathbf{D}_s$ is expressed in logarithmic scale and the arrows are tangential to the field streamlines.}
		\label{figure:TC1:section}
	\end{figure}

	\begin{figure}
		\centering
		\subfloat[3D-3D problem\label{figure:TC1:3d}] {\includegraphics[width=.3\textwidth]{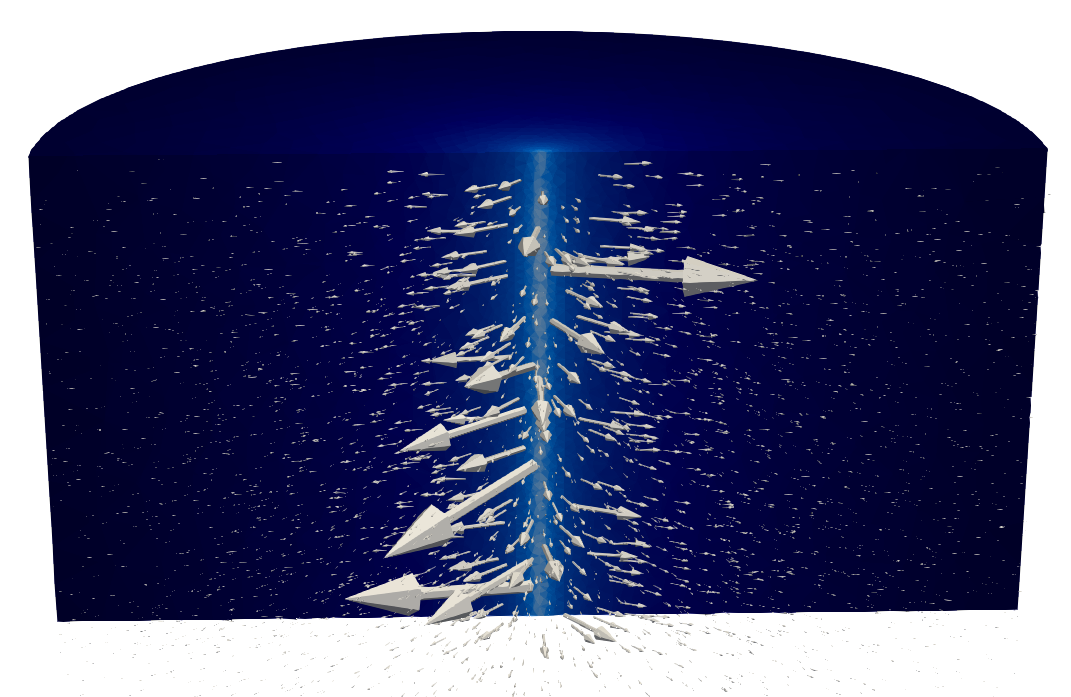}\quad}
		\subfloat[3D-1D problem\label{figure:TC1:1d}] {\includegraphics[width=.3\textwidth]{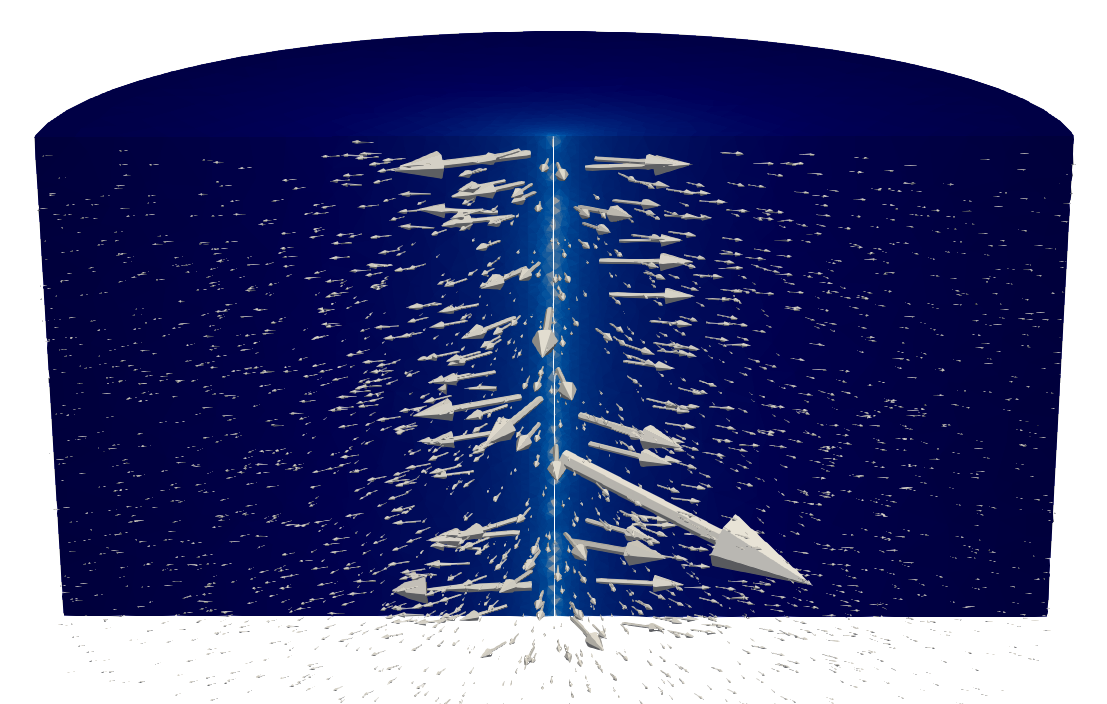}}
		\subfloat{\includegraphics[width=.08\textwidth]{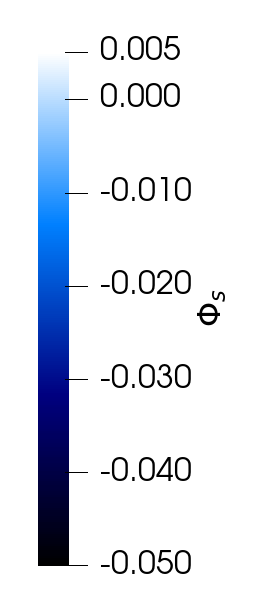}}
		\caption{\textbf{TC1 -} Computed potential $\Phi_s$ and electric field $\mathbf{D}_s$ in the 3D dielectric domain on a longitudinal section of $\Omega$. Comparison between the result on the original equi-dimensional 3D problem and on the mixed-dimensional 1D-3D domain.}
		\label{figure:TC1:3d1d}
	\end{figure}
	
	Finally, knowing the exact solution, we can compute the error committed in the approximation of the electric potential in the dielectric domain, in terms of the $L^2$ norm on $\Omega$, computed as $\text{error} = \|\Phi-\Phi_{ex}\|_{L^2(\Omega)}$. We want to investigate the dependence of this error on the radius $R$ of the initial equi-dimensional coupled problem. 
	As $R$ decreases, tending to zero, the starting three-dimensional cylindrical gas domain $\Omega_g$ tends to collapse on the centerline $\Lambda$ and we would expect the solution given by the mixed-dimensional problem to approximate better and better the exact solution to the original problem. \\ Figure~\ref{figure:TC1:validation} shows the plot of the $L^2$ error on the coarse mesh discussed above. We can see that the error decreases with $R$, and a linear dependence is observed. This plot provides us with a validation for the proposed geometrical reduction of the model.
	
	\begin{figure}
		\centering
		\begin{tikzpicture}
			\begin{loglogaxis}
				[title={Model validation},
				scale=.5,
				xlabel={$R$},
				ylabel={$\| \Phi - \Phi^{\text{ex}}  \|_{L^2(\Omega)}$},
				xtick={},
				ytick={},
				legend pos=north west,
				ymajorgrids=true,
				xmajorgrids=true,
				grid style=dashed,
				legend style = {draw=black, font=\scriptsize},
				legend entries ={
					$L^2$-error, $R$},
				scale only axis = true,
				width=0.5\textwidth,
				height=0.4\textwidth,
				]
				\addplot[color=darkBlue, mark=star, mark size = 3.2,
				line width = 1] coordinates {
					(0.000001,3.7848e-07)
					(0.00001,1.50974e-06)
					(0.0001,7.92641e-06)
					(0.001,0.000305378)
					(0.01,0.00543034)
				};
				\addplot[color=darkRed, dotted,
				line width = 1] coordinates {(1e-6,1e-6)(1e-2,1e-2)};	
			\end{loglogaxis}
		\end{tikzpicture}
		\caption{\textbf{TC1 -} $L^2$ error on the numerical approximation of the potential $\Phi$ in the dielectric domain when the radius $R$ of the original 3D gas domain $\Omega_g$ decreases from $10^{-2}$ to $10^{-6}$.}
		\label{figure:TC1:validation}
	\end{figure}
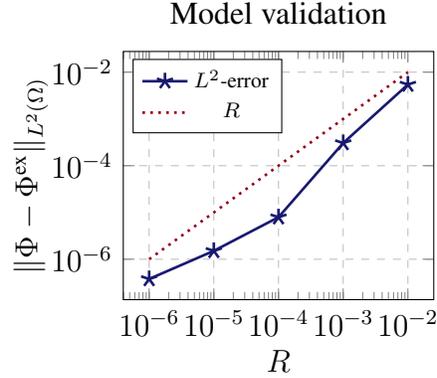

	\subsection{TC2: Cylinder and immersed straight line}
	\label{section:TC2}
	As second test case we observe the solution to the problem on a mixed-dimensional domain made of a cylinder and a straight line having one end on the bottom basis and the other one inside the three-dimensional volume. On the immersed end of the 1D domain we imposed a null-flux condition for the potential $\Phi_\Lambda$, and we consider the same sources and boundary data as in problem~\eqref{eq:TC1}.
	
	We can observe in Figure~\ref{figure:TC2:front} that the potential still has a radial decrease from the charged line towards the lateral surface of the cylinder and the region where it is most intense is very close the charged line. This is due to both the geometry of the one-dimensional domain and to the boundary effect related to the Dirichlet condition on the bottom basis, as discussed in the previous section.\\
	Figure~\ref{figure:TC2:lines} represents the streamlines of the electric field and its magnitude. The magnitude presents a steep increase near the one-dimensional domain, reflecting the singularity produced by a charged line, while the streamlines have the same radial direction as in Figure~\ref{figure:TC1:lines} close to the bottom, but tend to describe hyperbolas with smaller and smaller amplitude closer to the tip, as we would expect the electric field generated by a finite charged line to be. In this Figure we can also notice the importance of a sufficient mesh refinement also in the region above the one-dimensional domain, in order to capture this behavior.
	\begin{figure}
		\centering\includegraphics[width=.25\textwidth]{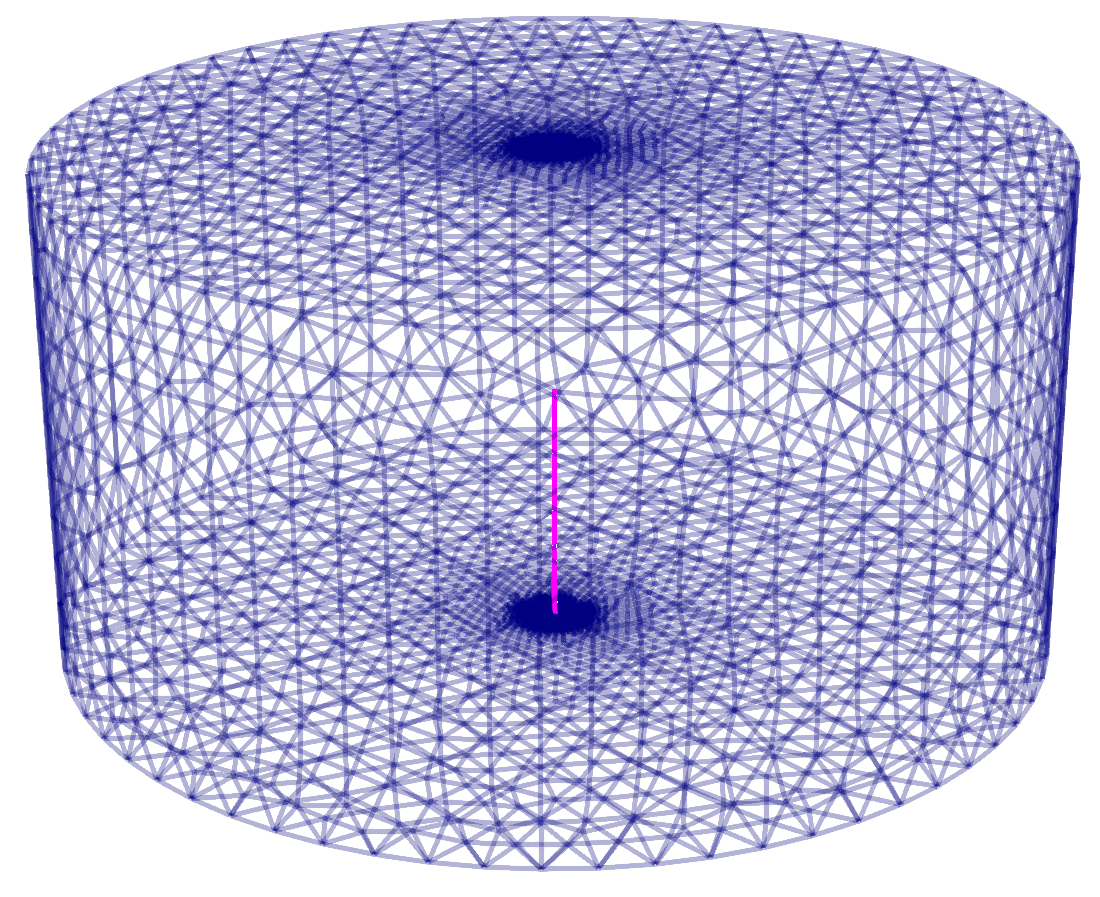}
		\caption{\textbf{TC2 -} Discretizations of the three-dimensional (blue) and one-dimensional (purple) domains.}
		\label{figure:TC2:mesh}
	\end{figure}

	\begin{figure}
		\centering
		\subfloat[Electric potential\label{figure:TC2:front}] {\includegraphics[width=.4\textwidth]{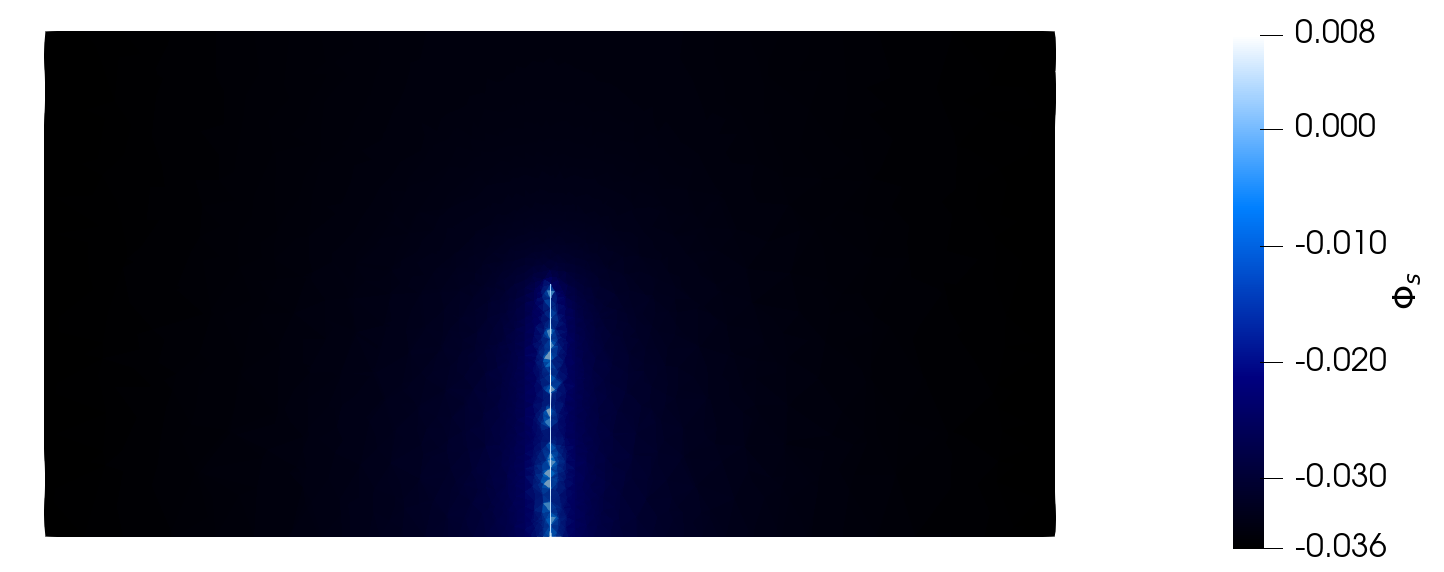}\quad}
		\subfloat[Electric field\label{figure:TC2:lines}] {\includegraphics[width=.4\textwidth]{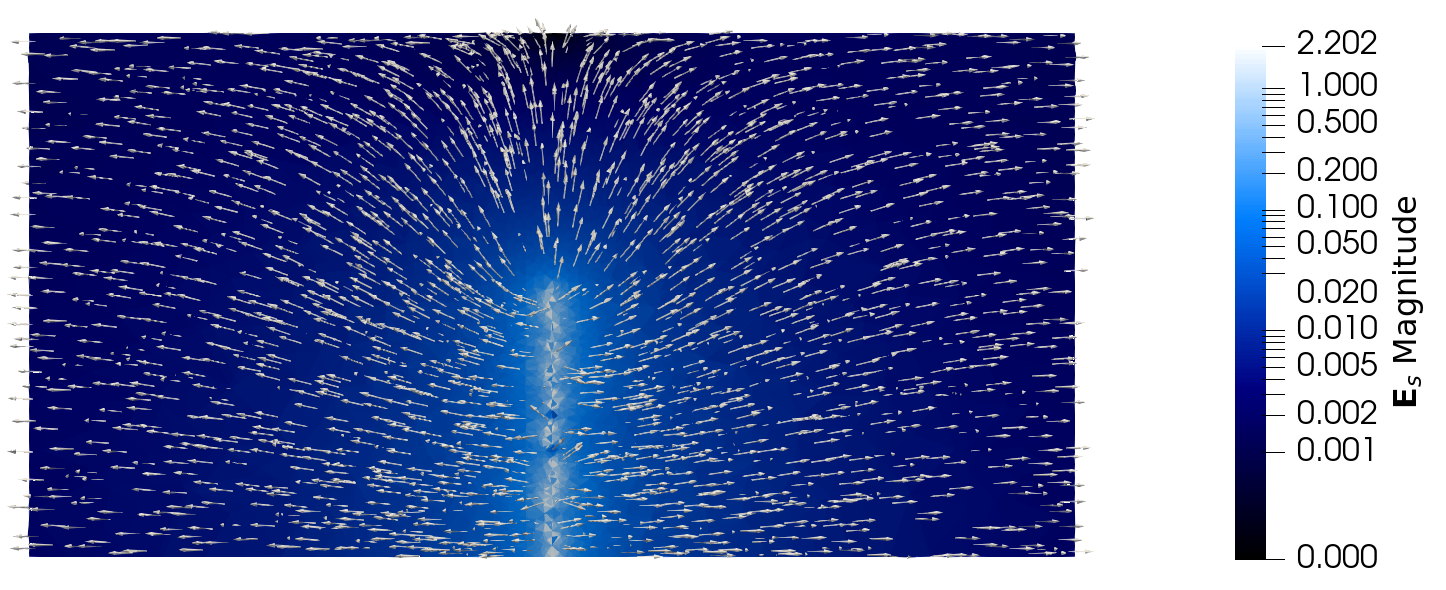}}
		\caption{\textbf{TC2 -} Computed potential $\Phi_s$ and electric field $\mathbf{D}_s$ in the 3D dielectric domain on a longitudinal section of $\Omega$. The magnitude of $\mathbf{D}_s$ is expressed in logarithmic scale and the arrows are tangential to the field streamlines.}
		\label{figure:TC2:section}
	\end{figure}
	
	\subsection{TC3: Electrical treeing}
	\label{section:TC3}
	The final test we present is made on a ramified domain representing the reduction to a one-dimensional graph of a typical electrical treeing, immersed a cylindrical domain (see Figure~\ref{figure:TC3:mesh}). The discretization is based on a coarse three-dimensional grid, refined in a co-axial cylinder enclosing the treeing. For the one-dimensional discretization, we simply take as mesh segments the edges of the graph. This realistic geometry was experimentally obtained from an existing defect in an electric cable: the 3D electrical treeing was detected via X-ray computed tomography \cite{schurch20193d, schurch2015comparison, schurch2014imaging} and the 1D structure was extracted as its skeleton. It is made of 12544 segments and is coupled with a tetrahedral grid on the cylinder composed of 50859 elements. 
	
	In Figure~\ref{figure:TC3:front}, representing a longitudinal section of the cylinder, we can observe that the potential is more intense closer to regions of the cylinder with high concentration of 1D edges, as expected. The electric field on the same section is displayed in Figure~\ref{figure:TC3:lines}, where the highest magnitude is observed in proximity of the ramification and the direction is radial with respect to segments, and curved exiting the tips, in agreement with Test Case 2.
	
	The numerical solution on these grids was computed by a C++ parallel implementation of the solver presented in Section~\ref{section:NumericalMethods}. The implementation is based on the Morgana complex modelling code~\cite{morgana} and relies on the Trilinos library \cite{trilinos-website}.
	The computational time requested for the parallel solution on six cores of a laptop with 16 GiB RAM, Intel(R) Core(TM) i5-9600K CPU, was approximately 5 minutes and 40 seconds, with 266 GMRES iterations with a tolerance $10^{-10}$ on the residual. A speedup could be obtained by employing a more ad hoc preconditioner, reducing the number of iterations of the solver. This is a noteworthy result, considering that such an extended defect could hardly be discretized in three-dimensions, due to its geometrical complexity, while thanks to the mixed-dimensional reduction we were able to solve the problem in reasonable time.
	\begin{figure}
		\centering\includegraphics[width=.25\textwidth]{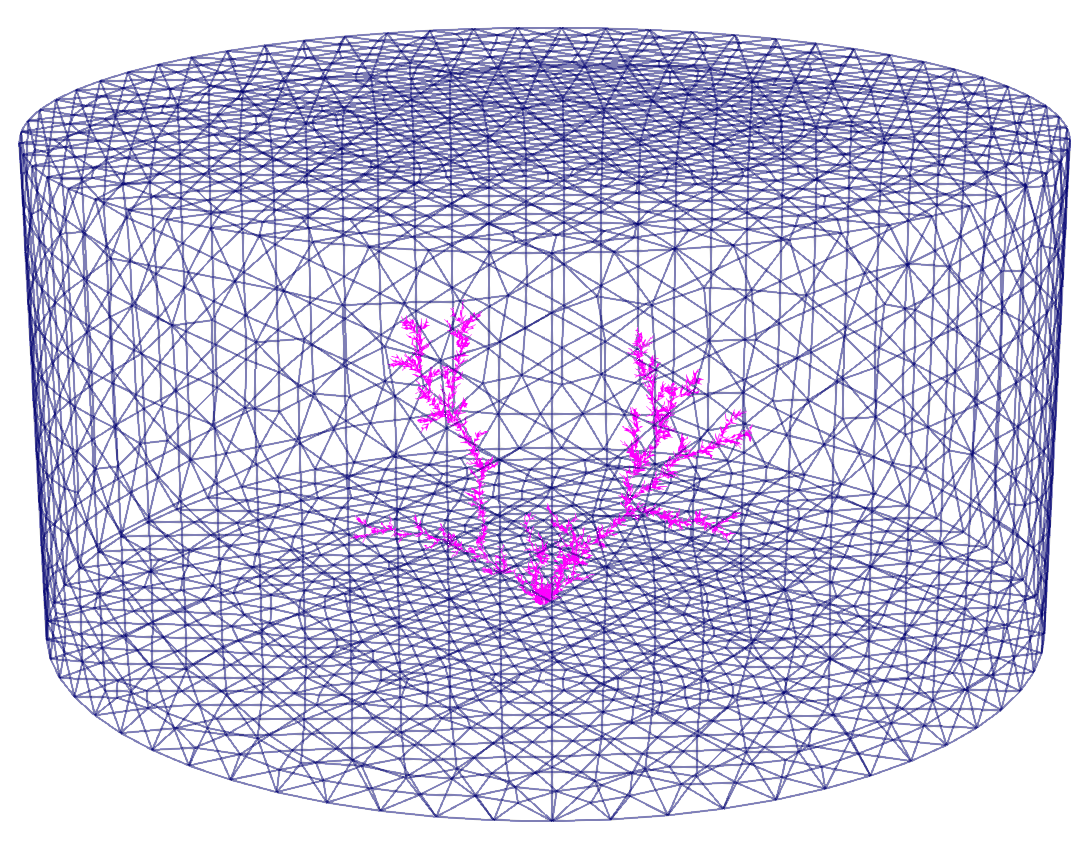}
		\caption{\textbf{TC3 -} Discretizations of the three-dimensional (blue) and one-dimensional (purple) domains.}
		\label{figure:TC3:mesh}
	\end{figure}

	\begin{figure}
		\centering
		\subfloat[Electric potential\label{figure:TC3:front}] {\includegraphics[width=.4\textwidth]{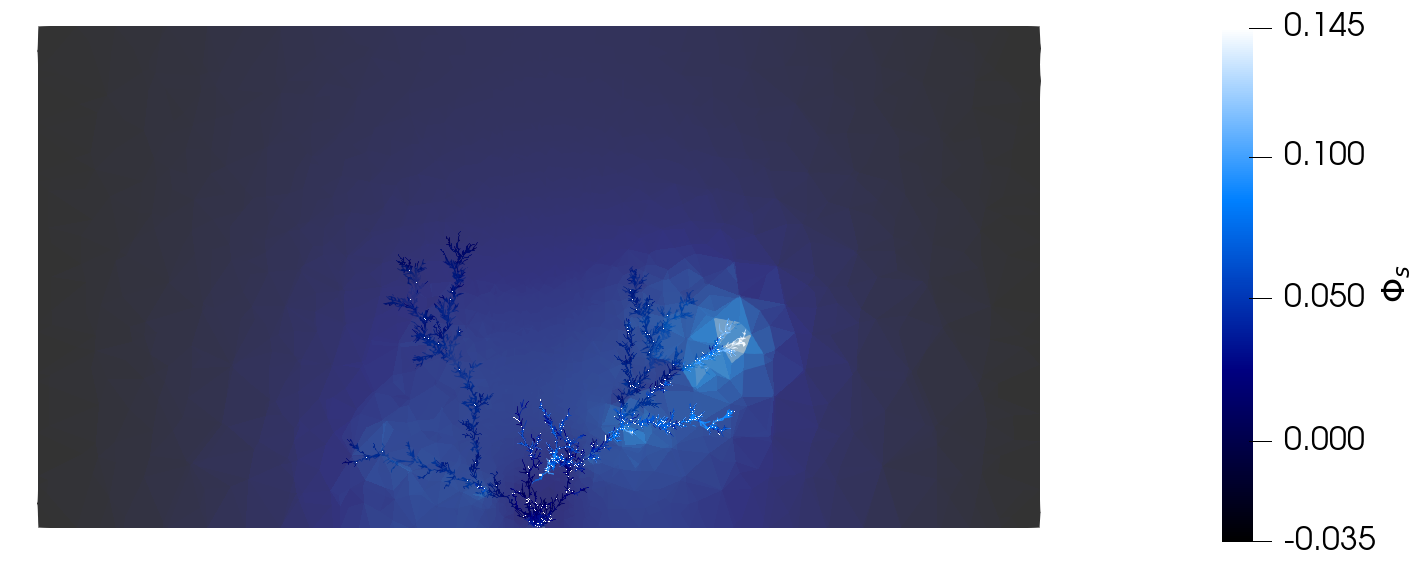}\quad}
		\subfloat[Electric field\label{figure:TC3:lines}] {\includegraphics[width=.4\textwidth]{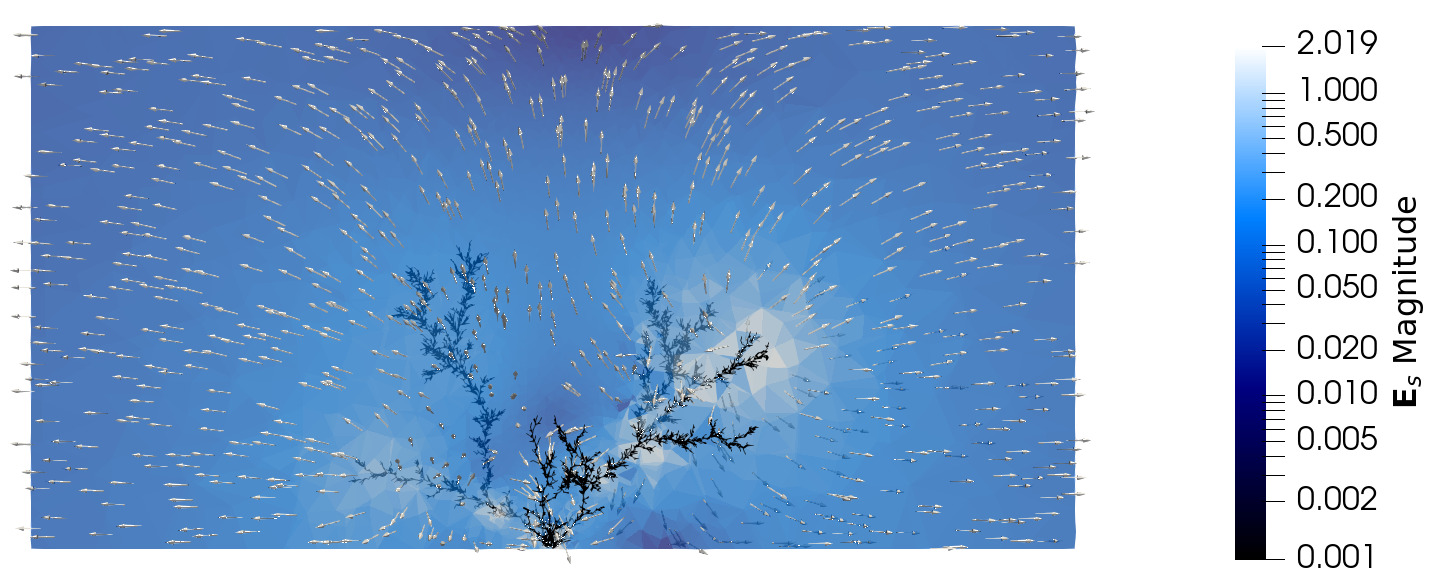}}
		\caption{\textbf{TC3 -} Computed potential $\Phi_s$ and electric field $\mathbf{D}_s$ in the 3D dielectric domain on a longitudinal section of $\Omega$. The magnitude of $\mathbf{D}_s$ is expressed in logarithmic scale and the arrows are tangential to the field streamlines.}
		\label{figure:TC3:section}
	\end{figure}

	\section{Conclusions}
	\label{section:Conclusions}
	Starting from the coupled three-dimensional electrostatic problem in two domains with different dielectric constants, we have deduced a reduced mixed-dimensional model describing the evolution of electric field and potential on a one-dimensional domain embedded in a large three-dimensional one. This problem is relevant because, together with the drift and diffusion of charged particles in the dielectric domain, and the chemical reactions, it models the partial discharges occurring inside insulating components of electric cables, and leading to the formation of electrical trees and their eventual deterioration. We have proven the well-posedness of the resulting continuous problem and solved it numerically with Finite Elements.\\
	We managed to conserve in the reduced problem some tricky properties of the electric field, such as the jump of its normal component across the interface between the two domains, and to incorporate in a natural way the interface conditions in the coupling terms. Moreover, we only required a reasonable assumption on the concentration of charge on each section of the gas domain and derived the corresponding profile of an electric potential produced by it. This way, we are not forced to approximate the potential as constant on sections, and consequently lose information on its profile and disregard the continuity condition on the interface.	\\
	We have validated the reduced model by comparing the approximated and exact solution on a simple geometry as the radius of the original three-dimensional gas domain decreases (Section~\ref{section:TC1}) and then qualitatively observed the solution on more complex geometries, such as an actual electrical treeing.
	The geometrical reduction proves crucial in more realistic cases: it allowed us to simulate this on a standard laptop in about half and hour, while the three-dimensional mesh of the defect could hardly be created.\\
	In order to further reduce computing time future work could address the improvement of the preconditioner proposed in Section~\ref{section:NM:solver}, since a more tailored one would make the GMRES require less iterations before convergence. Future work also includes the possibility to explore the Extended Finite Element Methods (XFEM) to better fit the singularity of the three-dimensional fields around the lines and a coupling with time-dependent evolution of the charge densities in the gas, requiring a model reduction as well.
	
	\newpage
	\printbibliography[
	heading=bibintoc,
	title={References}
	]
	
\end{document}